\documentclass[10pt,reqno]{amsart}
\usepackage[a4paper,top=2cm,right=1.2cm,left=1.2cm,bottom=2cm,footskip=20pt]{geometry}
\usepackage[hidelinks]{hyperref}
\pdfoutput=1
\usepackage{multicol,amssymb,amsaddr,enumitem,float,microtype,xcolor,tikz}
\usepackage[export]{adjustbox}
\usepackage[utf8]{inputenc}
\usepackage[italian]{babel}
\usepackage[margin=0cm,justification=centering]{caption}

\setenumerate{label={(\arabic*)},itemsep=2pt,topsep=0pt,leftmargin=1.7em}
\setitemize{itemsep=2pt,topsep=0pt,leftmargin=1.7em}
\usetikzlibrary{calc,decorations.markings,arrows.meta,intersections,decorations.pathreplacing}
\allowdisplaybreaks[4]
\newcommand{\alert}{\textbf}

\newcommand{\doi}[1]{\url{https://doi.org/#1}}
\newcommand{\isbn}[1]{\url{https://isbnsearch.org/isbn/#1}}
\newcommand{\arxiv}[1]{\href{https://arxiv.org/abs/#1}{preprint arXiv:#1}}
\newcommand{\web}[1]{\url{#1}}

\renewcommand{\emptyset}{\varnothing}
\numberwithin{equation}{section}
\numberwithin{figure}{section}

\newtheorem{thm}{Teorema}[section]

\newtheorem{prop}[thm]{Proposizione}
\newtheorem{cor}[thm]{Corollario}
\newtheorem{df}[thm]{Definizione}
\newtheorem{ex}[thm]{Esempio}

\newtheorem{pro}[thm]{Problema}

\makeatletter
\let\c@equation\c@figure
\makeatother

\newcommand{\R}{\mathbb{R}}
\newcommand{\C}{\mathbb{C}}
\newcommand{\Z}{\mathbb{Z}}
\newcommand{\N}{\mathbb{N}}

\newcommand{\wang}[9]{
\coordinate (o) at ({#1});
\coordinate (a) at ($(o)+(-0.5,-0.5)$);
\coordinate (b) at ($(o)+(-0.5,0.5)$);
\coordinate (c) at ($(o)+(0.5,0.5)$);
\coordinate (d) at ($(o)+(0.5,-0.5)$);

\filldraw[{#2},draw=black,very thin] (a) -- (o) -- (b);
\filldraw[{#3},draw=black,very thin] (b) -- (o) -- (c);
\filldraw[{#4},draw=black,very thin] (c) -- (o) -- (d);
\filldraw[{#5},draw=black,very thin] (d) -- (o) -- (a);

\draw (a) -- (b) -- (c) -- (d) -- cycle;

\begin{scope}[font=\small]
\draw (barycentric cs:a=1,o=1.3,b=1) node {\textsf{#6}};
\draw (barycentric cs:b=1,o=1.3,c=1) node {\textsf{#7}};
\draw (barycentric cs:a=1,o=1.3,d=1) node {\textsf{#8}};
\draw (barycentric cs:c=1,o=1.3,d=1) node {\textsf{#9}};
\end{scope}
}

\newcommand{\wangA}[1]{\wang{#1}{cyan!50}{green!50}{cyan!50}{green!50}{0}{1}{1}{0};}

\newcommand{\wangB}[1]{\wang{#1}{cyan!50}{green!50}{red!50}{green!50}{0}{1}{1}{2};}

\newcommand{\wangC}[1]{\wang{#1}{red!50}{green!50}{cyan!50}{green!50}{2}{1}{1}{0};}

\newlength{\raggio} 
\newlength{\mra}
\newlength{\bmra}
\newlength{\golden}
\colorlet{TA}{cyan!20}
\colorlet{TB}{green!20}
\colorlet{SA}{red!20}
\colorlet{SB}{yellow!20}
\colorlet{GTA}{white}
\colorlet{GTB}{gray!60}
\colorlet{GSA}{white}
\colorlet{GSB}{gray!60}
\colorlet{hat}{gray!40}
\definecolor{shirt}{RGB}{27,116,146}

\makeatletter
\def\@tocline#1#2#3#4#5#6#7{\relax
  \ifnum #1>\c@tocdepth 
  \else
    \par \addpenalty\@secpenalty\addvspace{#2}%
    \begingroup \hyphenpenalty\@M
    \@ifempty{#4}{%
      \@tempdima\csname r@tocindent\number#1\endcsname\relax
    }{%
      \@tempdima#4\relax
    }%
    \parindent\z@ \leftskip#3\relax \advance\leftskip\@tempdima\relax
    \rightskip\@pnumwidth plus4em \parfillskip-\@pnumwidth
    #5\leavevmode\hskip-\@tempdima
      \ifcase #1
       \or\or \hskip 1em \or \hskip 2em \else \hskip 3em \fi%
      #6 \hskip 0.5em \nobreak\relax
    \dotfill\hbox to\@pnumwidth{\@tocpagenum{#7}}\par
    \nobreak
    \endgroup
  \fi}
\makeatother

\begin{document}

\title{\vspace*{-1cm}Ordine privo di periodicità: il fascino \\[5pt] matematico delle tassellazioni}

\author[F.~D'Andrea]{\vspace*{-5mm}Francesco D'Andrea}

\address{Dipartimento di Matematica e Applicazioni ``R.~Caccioppoli'' \\ Universit\`a di Napoli Federico II \\
Complesso MSA, Via Cintia, 80126 Napoli, Italy}

\begin{abstract}
In questo saggio ripercorriamo alcuni risultati celebri della teoria delle tassellazioni.
Si tratta ovviamente di un argomento vastissimo, ed una trattazione esaustiva richiederebbe vari libri. Qui ci concentriamo sulle tassellazioni del piano, in particolare quelle aperiodiche, dai primi celebri esempi di tassellazioni di Wang e di Penrose, fino alla recente scoperta di un monotassello aperiodico (Smith et al., 2023). La scelta degli argomenti (e delle dimostrazioni riportate in dettaglio) è basata su gusti personali ed estetici, oltre che di carattere didattico.
\end{abstract}

\maketitle

\begin{multicols}{2}

\section{Introduzione}

Il problema di ricoprire una superficie piana con forme geometriche, senza lasciare spazi vuoti e senza sovrapposizioni, affascina da sempre matematici, artisti e artigiani.
Un tale ricoprimento è detto \emph{piastrellatura}, \emph{mosaico} o \emph{tassellazione}. 
Le forme di base di cui è composto sono dette \emph{tessere} o \emph{tasselli}.
Nella vita di tutti i giorni siamo continuamente circondati da tassellazioni, ad iniziare dai rivestimenti di pavimenti e pareti, che possono rappresentare elementi decorativi importanti (si pensi alle maioliche in ceramica --- o \emph{riggiole} in napoletano --- per le quali sono famose la penisola sorrentina e la costiera amalfitana). Tassellazioni si possono ammirare nei mosaici e nelle vetrate delle chiese. Tasselli dalla forma di animali stilizzati sono ricorrenti nell'opera di M.C.~Escher.
In matematica, un primo studio sistematico delle tassellazioni del piano si può trovare nell'opera di Keplero, \textit{Harmonices mundi} (1619).

Parte del fascino matematico delle tassellazioni deriva dal fatto che molti suoi problemi si possono formulare in un linguaggio comprensibile da tutti, anche se la soluzione poi a volte richiede tecniche tutt'altro che elementari. In effetti, possiamo citare almeno due matematici amatoriali che hanno dato un contributo fondamentale alla teoria. La prima è la casalinga californiana Marjorie Rice, che scoprì tassellazioni del piano con pentagoni che erano sfuggite ai matematici professionisti \cite{Sch78}. Più di recente David Smith, un tecnico della stampa in pensione, è diventato celebre per aver scoperto il primo esempio di monotassello aperiodico \cite{SMKG24a}.

In questo saggio ripercorriamo alcuni risultati celebri della teoria delle tassellazioni.
Inizieremo in Sezione~\ref{sec:ts} con alcune definizioni generali ed una discussione sulle simmetrie delle tassellazioni, e vedremo un primo esempio di tassellazione aperiodica costruita con il metodo noto come ``cut-and-project''. In Sezione~\ref{sec:pdg} useremo l'esempio dei puzzle e del gioco del domino per introdurre la nozione di ``regola di corrispondenza'' e spiegare il legame fra tassellazioni unidimensionali e grafi.
In Sezione~\ref{sec:Wang}, studieremo le tassellazioni di Wang, che sono una specie di versione bidimensionale del domino.
Passeremo quindi alle celebri tassellazioni di Penrose, iniziando con il loro studio
attraverso i triangoli di Robinson in Sezione~\ref{sec:pen1}, per poi parlare delle pentagriglie di de Brujin in Sezione~\ref{sec:pen2}. La Sezione~\ref{sec:simdin} è una breve digressione sul legame fra tassellazioni e dinamica simbolica. Nella Sezione~\ref{sec:einstein} parleremo infine brevemente del ``monotassello'' di David Smith, ed in Sezione~\ref{sec:conc} trarremo alcune conclusioni.

\section{Tassellazioni e simmetrie}\label{sec:ts}

Iniziamo con la definizione matematica esatta di tassellazione.
Per noi il \emph{piano} è quello della geometria cartesiana, l'insieme $\R^2$ di tutte le coppie di numeri reali (ascissa e ordinata di un punto in un sistema di riferimento bidimensionale).

\begin{df}
Un \alert{tassello} è un sottoinsieme chiuso di $\R^2$ con interno non vuoto.
Una famiglia $T:=\{t_i\}_{i\in I}$ di tasselli (non necessariamente finita o numerabile)
è detta \alert{tassellazione} se gli interni dei tasselli hanno intersezione vuota, in formule:
\[
\mathring{t}_i \cap\mathring{t}_j=\emptyset\;\forall\;i\neq j.
\]
Chiamiamo \alert{supporto} di $T$ l'unione $\cup_{i\in I}t_i$ dei suoi tasselli. Una \alert{tassellazione del piano} è una tassellazione il cui supporto è l'intero insieme $\R^2$.
\end{df}

Il lettore che non abbia familiarità con la topologia, può immaginare che i tasselli siano poligoni (sarà così in tutti gli esempi più importanti che discuteremo).
La nozione di tassellazione ha senso per uno spazio topologico arbitrario, ma a noi interesserà solo il caso di $\R^2$.

In Figura~\ref{fig:alveare}, in nero vediamo (una porzione di) una tassellazione del piano
con esagoni regolari, anche detta \emph{a nido d'ape}.
Notiamo che i tasselli sono tutti dello stesso tipo, più precisamente sono tutti congruenti (si possono trasformare l'uno nell'altro con una isometria).

\begin{figure}[H]
\includegraphics[page=4,width=0.8\columnwidth]{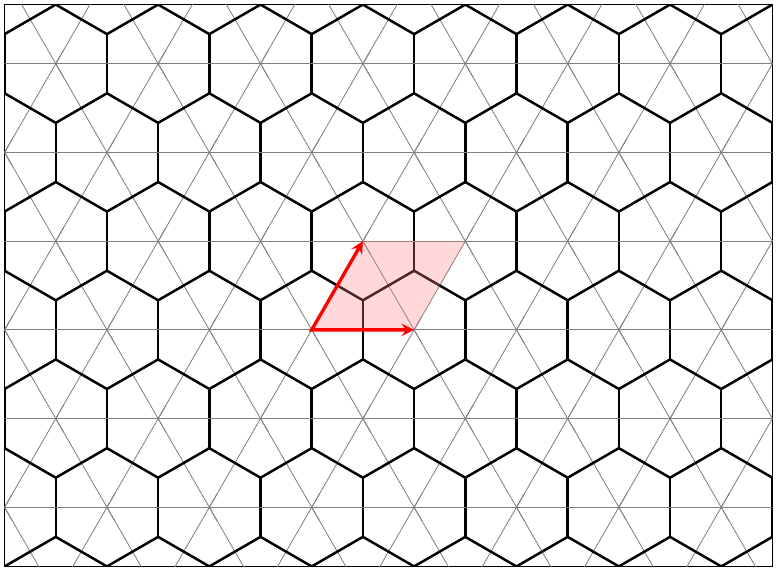}
\caption{Tassellazione esagonale.}\label{fig:alveare}
\end{figure}

\begin{df}
La classe di congruenza di un tassello e detta \alert{prototassello}.
La famiglia di prototasselli di una tassellazione è detto il suo \alert{protoinsieme}.
\end{df}

Un prototassello è quindi una forma geometrica che non ha una posizione ben definita nel piano. Ogni volta che lo posizioniamo nel piano, otteniamo uno specifico tassello. La tassellazione in Figura~\ref{fig:alveare} ha un protoinsieme formato da un unico prototassello: un esagono regolare.

Un'altra cosa che possiamo osservare in Figura~\ref{fig:alveare} sono, in grigio, tre fasci di rette parallele, passanti per i centri degli esagoni ed ortogonali ai suoi lati. Se trasliamo la figura in direzione di una qualsiasi di queste rette, di un passo pari all'altezza di un esagono, la figura non cambia. Possiamo scegliere due direzioni indipendenti, ad esempio quelle date dai vettori rossi in figura.
Assumendo che il punto di applicazione dei vettori rossi in figura sia l'origine di $\R^2$, vediamo che i centri degli esagoni formano un \emph{reticolo}, un sottoinsieme $\Z\vec{v}_1+\Z\vec{v}_2\subseteq\R^2$ dato da tutte le combinazioni intere di due vettori $\vec{v}_1$ e $\vec{v}_2$ linearmente indipendenti (nel caso in esame, $\vec{v}_1$ e $\vec{v}_2$ sono i due vettori rossi in figura).

Ogni reticolo definisce una tassellazione del piano, i cui tasselli sono le sue \emph{celle elementari}, ovvero i parallelogrammi che per lati hanno i segmenti che uniscono due vertici consecutivi del reticolo (ad esempio il parallelogramma rosso chiaro in Figura~\ref{fig:alveare}).

Una isometria $f:\R^2\to\R^2$ che lascia invariata una tassellazione $T$, ovvero tale che $f(T)=T$, è detta \emph{simmetria} della tassellazione stessa. Ricordiamo che ogni isometria si può ottenere componendo traslazioni, rotazioni e riflessioni. La tassellazione in Figura~\ref{fig:alveare} è invariante per traslazioni di parametro dato da un qualsiasi vettore del reticolo generato dai vettori rossi, per rotazioni di 60$^\circ$ attorno al centro di un qualsiasi esagono e di 120$^\circ$ attorno ad uno qualsiasi dei suoi vertici, ed è invariante per riflessione rispetto a numerose rette, fra le quali quelle in grigio in figura.

\smallskip

\begin{df}
Una tassellazione del piano è detta:\vspace{-3pt}
\begin{itemize}\itemsep=1pt
\item \alert{periodica} se esiste una traslazione non nulla che la lascia invariata,
\item \alert{aperiodica} se non è periodica,
\item \alert{bi-periodica} se esistono due traslazioni in direzioni linearmente indipendenti che la lasciano invariata.
\end{itemize}
Infine, diciamo che una tassellazione ha \alert{simmetria $n$-fold} se è invariante per rotazioni di un angolo di $2\pi/n$.
\end{df}

La tassellazione a nido d'ape è bi-periodica ed ha simmetria $6$-fold. 
Un noto teorema di cristallografia afferma che:

\begin{thm}[di restrizione cristallografica]
Se $f$ è una simmetria rotazionale di un reticolo $\Lambda$ di $\R^2$ (cioè soddisfa $f(\Lambda)=\Lambda$), allora $f$ è $n$-fold con $n\in\{1,2,3,4,6\}$.
\end{thm}

\begin{proof}
La dimostrazione è un semplice esercizio di algebra lineare.
In formule, una rotazione di centro $\vec{c}\in\R^2$ ha la forma
\[
f(\vec{v})=R(\vec{v}-\vec{c})+\vec{c}
\]
dove
\[
R:=\bigg(\!\begin{array}{rr}\cos\theta & -\sin\theta \\ \sin\theta & \cos\theta\end{array}\!\bigg) ,\qquad 0<\theta\leq 2\pi,
\]
è la matrice di rotazione. Sia $\Lambda=\Z\vec{v}_1+\Z\vec{v}_2$ un reticolo e
supponiamo che $f(\Lambda)=\Lambda$. Scrivendo i vettori di $\R^2$ in colonna formiamo la matrice $2\times 2$ reale:
\[
V:=\Big(\;\vec v_1\;\Big|\;\vec v_2\;\Big) .
\]
Per l'ipotesi di invarianza, $f(\vec{v}_1)-f(\vec{0})=R\vec{v}_1$ ed $f(\vec{v}_2)-f(\vec{0})=R\vec{v}_2$ sono combinazioni lineari di $\vec v_1$ e $\vec v_2$ a coefficienti interi. In termini di matrici,
\[
RV=VA
\]
per una qualche matrice $A$ intera $2\times 2$. Ma $V$ è invertibile ($\vec v_1$ e $\vec v_2$ sono linearmente indipendenti). Siccome $R$ è coniugata ad una matrice con traccia intera, la traccia di $R$, ovvero $2\cos\theta$, deve essere un intero.
Quindi $\cos\theta\in\left\{0,\frac{1}{2},-\frac{1}{2},\pm 1\right\}$, che corrisponde a
$\theta\in\{\pm\pi/2,\pm\pi/3,\pm 2\pi/3,\pi,2\pi\}$.
\end{proof}

Un teorema analogo vale per tassellazioni del piano periodiche, se almeno un tassello è limitato (cioè contenuto in un disco di raggio sufficientemente grande). Per tali tassellazioni, le uniche simmetrie rotazionali ammesse sono $n$-fold con $n\in\{1,2,3,4,6\}$ (si veda ad esempio \cite[Sez.~2.3]{DanPen} per la dimostrazione). L'ipotesi di esistenza di un tassello limitato serve ad escludere esempi patologici come la tassellazione con un singolo tassello dato dall'intero piano: essa è bi-periodica con simmetria $n$-fold per ogni valore di $n$. Per un esempio di tassellazione (non periodica) con tasselli limitati e simmetria $5$-fold si può vedere \cite[Fig.~2.7]{DanPen}.

Mentre una tassellazione periodica con simmetria rotazionale può avere tanti centri di simmetria (come nell'esempio della tassellazione a nido d'ape), in una tassellazione aperiodica se esiste un centro di simmetria questo è anche unico. Questo si dimostra facilmente per assurdo: se $f_1$ ed $f_2$ sono rotazioni non banali con centri distinti, e sono entrambe simmetrie di una tassellazione $T$, allora anche la composizione $f_1f_2f_1^{-1}f_2^{-1}$ è una simmetria, e si può verificare che tale composizione è una traslazione non nulla, cioè $T$ è periodica.

Non è difficile costruire una tassellazione aperiodica: dividiamo ciascun esagono di una tassellazione a nido d'ape in modo da ottenere una tassellazione del piano bi-periodica con rombi (Figura~\ref{fig:romper}). Quindi capovolgiamo uno qualsiasi degli esagoni (in rosso in Figura~\ref{fig:romaper}), rompendo in questo modo la simmetria traslazionale.

\begin{figure}[H]
\includegraphics[page=1,width=0.8\columnwidth]{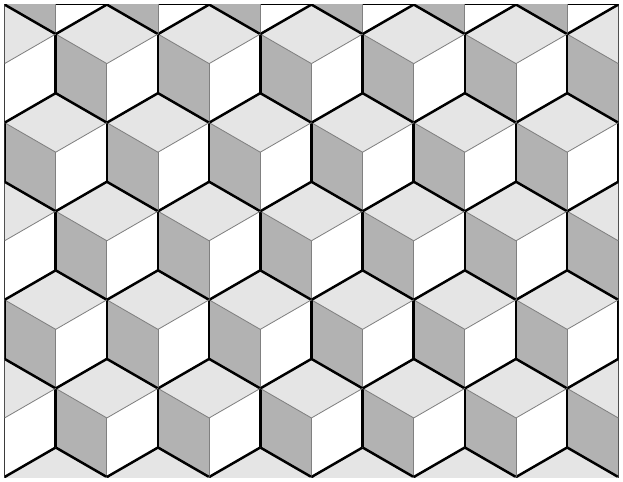}
\caption{Tassellazione periodica con rombi.}\label{fig:romper}
\end{figure}

\begin{figure}[H]
\includegraphics[page=4,width=0.8\columnwidth]{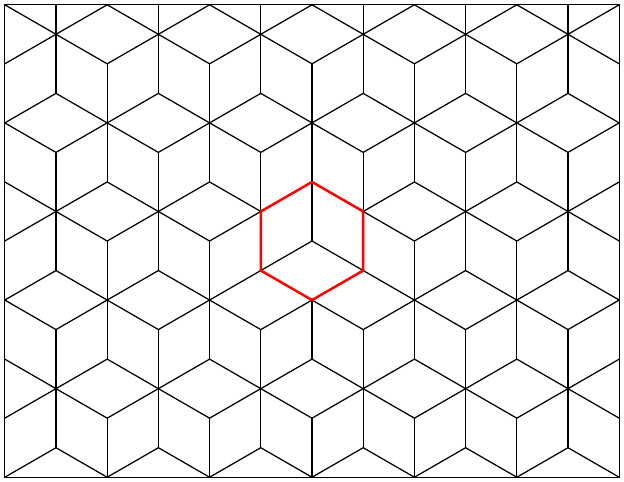}
\caption{Tassellazione aperiodica con rombi.}\label{fig:romaper}
\end{figure}

Questa tassellazione d'altronde non è particolarmente interessante. Vedremo a breve come costruire tassellazioni aperiodiche più interessanti.

Osserviamo come, in Figura~\ref{fig:romper}, la colorazione dei rombi suggerisca che la tassellazione si possa ottenere proiettando gli spigoli di cubi unitari.
Cosa succede cambiando la direzione di proiezione?
Scendiamo di una dimensione e studiamo ricoprimenti di una retta con segmenti di diversa lunghezza, e senza sovrapposizioni.
Partiamo da una tassellazione di $\R^2$ con quadrati unitari, le celle del reticolo $\Z^2$ e consideriamo la retta $\ell$ passante per l'origine e di coefficiente angolare $0<m<1$. In Figura~\ref{fig:staircase}, vediamo evidenziate in grigio le celle del reticolo che intersecano la retta $\ell$. Se proiettiamo ortogonalmente i vertici di tali quadrati che sono al di sopra della retta, otteniamo una tassellazione della retta con segmenti di due lunghezze: uno di lunghezza $1/\sqrt{1+m^2}$ (in celeste in figura), proiezione di un lato orizzontale, ed uno più corto di lunghezza $m/\sqrt{1+m^2}$, proiezione di un lato verticale di un quadrato.

Se $m=1$, chiaramente si ottiene una tassellazione periodica con segmenti tutti di uguale lunghezza.

\begin{figure}[H]
\includegraphics[width=0.9\columnwidth]{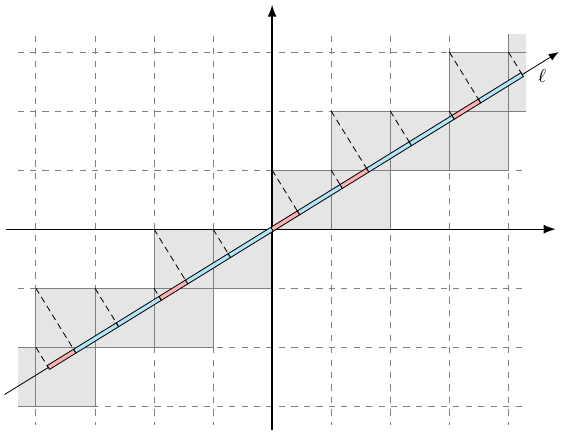}
\caption{Cut-and-project.}\label{fig:staircase}
\end{figure}

\begin{prop}
Se $m$ è irrazionale, la tassellazione di $\ell$ in Figura~\ref{fig:staircase} è aperiodica. 
\end{prop}

\begin{proof}
La proiezione ortogonale di $(x,y)\in\R^2$ su $\ell$ è un punto di coordinata monodimensionale
\[
f(x,y):=\frac{x+my}{\sqrt{1+m^2}}
\]
su $\ell$.
Indichiamo con $\Gamma\subset\Z^2$ l'insieme dei vertici che vengono proiettati per formare la tassellazione di $\ell$.
Per assurdo, assumiamo che l'insieme $f(\Gamma)$ sia periodico di periodo $\tau>0$, cioè resti invariato se traslato di $\tau$ lungo la retta $\ell$. Il periodo deve essere necessariamente una combinazione intera delle lunghezze dei segmenti di base, ovvero
\[
\tau=\frac{a + bm}{\sqrt{1+m^2}}=f(a,b) ,
\]
per un qualche $(a,b)\in\Z^2\smallsetminus\{(0,0)\}$. Se $f(\Gamma)$ è periodico di periodo $\tau$, si dovrebbe avere $k\tau\in f(\Gamma)$ per ogni $k\in\Z$ (questi sono i punti nell'orbita per traslazioni dell'origine). Ma
\[
k\tau=kf(a,b)=f(ka,kb) .
\]
Siccome $m$ è irrazionale, la restrizione di $f$ a $\Z^2$ è iniettiva, e da
$f(ka,kb)\in f(\Gamma)$ deduciamo che $P_k:=k(a,b)\in\Gamma$ per ogni $k\in\Z$.
La distanza di $P_k$ da $\ell$ cresce linearmente con $k$, e per $k$ sufficientemente grande $P_k$ non può essere l'estremo di una cella che interseca la retta $\ell$, arrivando in questo modo ad una contraddizione.
\end{proof}

Questo esempio mostra come si possa costruire una tassellazione aperiodica a partire da una periodica in dimensione più alta. Tale metodo è noto come \emph{cut-and-project}, ed è stato introdotto per la prima volta da
Y.~Meyer nel campo dell'analisi armonica \cite{Mey72}.
Quando $m$ è l'inverso della \emph{sezione aurea}, si ottiene una tassellazione, detta di Fibonacci, con proprietà molto simili a quelle delle tassellazioni di Penrose, che discuteremo più avanti. Ad esempio, ogni porzione finita di questa tassellazione è ripetuta infinite volte lungo la retta $\ell$ (diciamo che la tassellazione è \emph{ripetitiva}).

\section{Puzzle, domino e grafi}\label{sec:pdg}

Un aspetto non trascurabile della matematica è quello ricreativo. Vediamo cosa possiamo imparare sulle tassellazioni studiando il gioco del Tetris, i puzzle ed il domino.
Un attento giocatore avrà notato che ogni pezzo del Tetris è formato da quattro quadrati adiacenti, anche detti \emph{tetramini}. Una caratteristica dei tetramini è che due quadrati adiacenti condividono un intero lato.

Una tassellazione con poligoni è detta \emph{edge-to-edge} se ogni coppia di tasselli adiacenti condivide un intero lato.
Nella figura seguente, la tassellazione di sinistra è edge-to-edge, mentre quella di destra no:
\begin{center}
\begin{tikzpicture}[semithick,yscale=0.5]

\fill[brown!20] (0,0) rectangle (2,2);

\draw
(0,0) -- (2,0)
(0,1) -- (2,1)
(0,2) -- (2,2)
(0,0) -- (0,2)
(1,0) -- (1,2)
(2,0) -- (2,2);

\begin{scope}[xshift=3cm]
\fill[brown!20] (0,0) rectangle (2,1);
\fill[brown!20] (0.5,1) rectangle (2.5,2);

\draw
(0,0) -- (2,0)
(0,1) -- (2.5,1)
(0.5,2) -- (2.5,2)
(0,0) -- (0,1)
(1,0) -- (1,1)
(2,0) -- (2,1)
(0.5,1) -- (0.5,2)
(1.5,1) -- (1.5,2)
(2.5,1) -- (2.5,2);
\end{scope}

\end{tikzpicture}
\end{center}
In generale, una \emph{poliforma} è una tassellazione edge-to-edge con supporto connesso formata unendo poligoni tutti congruenti fra loro.
Un caso speciale di interesse per noi sono quadrilateri con coppie di lati consecutivi della stessa lunghezza. Un tale poligono è detto \emph{dardo} se concavo, e \emph{aquilone} se convesso:
\begin{equation}\label{eq:dardoaq}
\begin{tikzpicture}[semithick,font=\footnotesize,baseline=(current bounding box.center)]

\coordinate (v1) at (0,0);
\coordinate (v2) at (54:1.618);
\coordinate (v3) at (0,1.618);
\coordinate (v4) at (126:1.618);
\coordinate (v5) at (-3,1.618);
\coordinate (v6) at ($(v5)+(-54:1.618)$);
\coordinate (v7) at ($(v5)+(0,-1)$);
\coordinate (v8) at ($(v5)+(234:1.618)$);
\coordinate (v9) at ($(v5)+(0,-2.5)$);

\node[left] at (-3.2,-0.2) {Dardo};
\node[right] at (0.2,-0.2) {Aquilone};

\draw[gray!40,thick,smooth,tension=0.6] plot coordinates{(v1) (-0.2,-0.4) (0.2,-0.8) (0,-1.2)};

\draw[gray!40,thick] (v7) -- (v9);

\foreach \k in {0,...,3} \draw[thick,gray!40] ($(v9)+(0,0.15*\k)$) -- ++(-30:0.2) ($(v9)+(0,0.15*\k)$) -- ++(210:0.2);

\draw[fill=cyan!30] (v1) -- (v2) -- (v3) -- (v4) -- cycle;
\draw[fill=red!30] (v5) -- (v6) -- (v7) -- (v8) -- cycle;
\draw[dashed,gray] (v1) -- (v3) (v5) -- (v7);


\end{tikzpicture}
\end{equation}
Una poliforma ottenuta unendo aquiloni tutti uguali è detta \emph{poliaquilone}. 

Anche un puzzle è una specie di tassellazione. Tipicamente, i tasselli sono ottenuti da dei quadrati unitari deformandone leggermente i lati, in modo da avere delle sporgenze e delle rientranze. In Figura~\ref{fig:puzzle} vediamo una tassellazione con tessere di un puzzle il cui supporto è un rettangolo.

\begin{figure}[H]
\includegraphics[page=1]{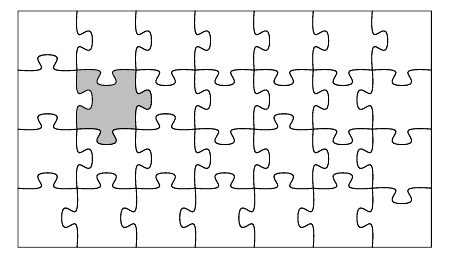}
\caption{Puzzle.}\label{fig:puzzle}
\end{figure}

I prototasselli di questa tassellazione sono i cinque pezzi seguenti:
\begin{center}
\includegraphics[page=2]{puzzle.pdf}
\end{center}
che possono essere posizionati traslati e ruotati in modo da formare il motivo in Figura~\ref{fig:puzzle}.
Osserviamo che l'unico modo per coprire lo spazio bianco che corrisponde alla rientranza di uno di questi tasselli è incastrandolo con un altro tassello. Con un abuso di terminologia, possiamo dire che una tassellazione con questi prototasselli è necessariamente edge-to-edge. Non solo, due tasselli possono essere adiacenti solo se il lato comune corrisponde ad una rientranza in uno dei tasselli e ad una sporgenza nell'altro.

Anzichè utilizzare tessere complicate, possiamo ottenere un gioco
equivalente usando quadrati con lati orientati.
Possiamo trasformare ogni sporgenza in un lato orientato in senso orario e ogni rientranza in un lato orientato in senso antiorario. Ai lati dritti non è associata alcuna orientazione. I cinque prototasselli 
del puzzle in Figura \ref{fig:puzzle} diventano:
\begin{center}
\newlength{\spazio}
\setlength{\spazio}{1.65cm}
\begin{tikzpicture}[>=Straight Barb]

\foreach \k in {1,...,5} \filldraw[semithick,draw=black,fill=gray!50] (\k*\spazio,0) rectangle +(1,1);

\begin{scope}[decoration={
    markings,
    mark=at position 0.5 with {\arrow[scale=1.1,xshift=1.5pt]{>}}}
    ]

\draw[semithick,postaction={decorate}] (\spazio,1) -- +(1,0);
\draw[semithick,postaction={decorate}] ($(\spazio,0)+(1,0)$) -- +(0,1);

\draw[semithick,postaction={decorate}] (2*\spazio,1) -- +(1,0);
\draw[semithick,postaction={decorate}] ($(2*\spazio,0)+(1,0)$) -- +(0,1);
\draw[semithick,postaction={decorate}] (2*\spazio,0) -- +(0,1);

\draw[semithick,postaction={decorate}] (3*\spazio,1) -- +(1,0);
\draw[semithick,postaction={decorate}] ($(3*\spazio,0)+(1,0)$) -- +(0,1);
\draw[semithick,postaction={decorate}] (3*\spazio,0) -- +(1,0);
\draw[semithick,postaction={decorate}] (3*\spazio,1) -- +(0,-1);

\draw[semithick,postaction={decorate}] (4*\spazio,1) -- +(1,0);
\draw[semithick,postaction={decorate}] ($(4*\spazio,0)+(1,1)$) -- +(0,-1);
\draw[semithick,postaction={decorate}] (4*\spazio,0) -- +(1,0);
\draw[semithick,postaction={decorate}] (4*\spazio,1) -- +(0,-1);

\draw[semithick,postaction={decorate}] (5*\spazio,0) -- +(1,0);
\draw[semithick,postaction={decorate}] ($(5*\spazio,0)+(1,0)$) -- +(0,1);
\draw[semithick,postaction={decorate}] ($(5*\spazio,0)+(1,1)$) -- +(-1,0);
\draw[semithick,postaction={decorate}] (5*\spazio,1) -- +(0,-1);

\end{scope}

\end{tikzpicture}
\end{center}
Possiamo quindi considerare quelle tassellazioni edge-to-edge ottenute unendo quadrati lungo lati con frecce di verso concorde.
Questo genere di restrizioni va sotto il nome di \emph{regole di corrispondenza}. Aggiungere delle decorazioni e delle regole permette di semplificare la forma dei tasselli senza cambiare il problema matematico oggetto di studio.
La Figura~\ref{fig:puzzle}, sostituendo sporgenze e rientranze con decorazioni come spiegato, diventa la tassellazione in Figura~\ref{fig:matching}.

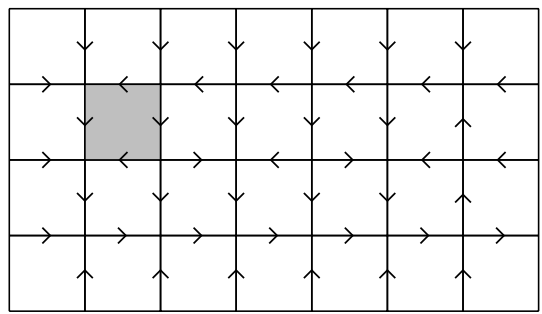
\begin{figure}[H]
\begin{tikzpicture}[semithick,>=Straight Barb]

\draw (0,0) grid (7,4);
\fill[gray!50] (1,2) rectangle +(1,1);

\begin{scope}[decoration={
    markings,
    mark=at position 0.5 with {\arrow[scale=1.1,xshift=1.5pt]{>}}}
    ]

\foreach \k in {1,...,6} {
	\draw[postaction={decorate}] (\k,4) -- +(0,-1);
	\draw[postaction={decorate}] (\k,1) -- +(1,0);
	\draw[postaction={decorate}] (\k,0) -- +(0,1);
	}

\foreach \j in {2,3} {
	\foreach \k in {1,...,5} {
		\draw[postaction={decorate}] (\k,\j) -- +(0,-1);
		}
	}

\foreach \j in {1,2} \draw[postaction={decorate}] (6,\j) -- +(0,1);

\foreach \j in {1,2,3} \draw[postaction={decorate}] (0,\j) -- +(1,0);

\foreach \k in {2,...,7} \draw[postaction={decorate}] (\k,3) -- +(-1,0);

\foreach \k in {2,4} \draw[postaction={decorate}] (\k,2) -- +(1,0);
\foreach \k in {2,4,6,7} \draw[postaction={decorate}] (\k,2) -- +(-1,0);

\end{scope}

\end{tikzpicture}
\caption{Regole di corrispondenza.}\label{fig:matching}
\end{figure}

Passiamo da tessere quadrate a coppie di quadrati, ovvero tessere del domino.
Immaginiamo di avere un certo numero di tessere del domino disposte orizzontalmente, e di volerle posizionare nel piano in modo da formare una tassellazione. Cambiamo però leggermente le regole del gioco e permettiamo che le tessere vengano traslate, ma non ruotate (in altri termini, cambiamo la nozione di prototassello in classe di equivalenza rispetto alle sole traslazioni).
Possiamo allora pensare una tessera del domino come una coppia $(l,r)$, dove $l,r\in\{1,\ldots,n\}$ rappresenta il numero di pallini nella casella di sinistra risp.\ destra (e, tipicamente, $n=6$). Nell'esempio
\begin{center}
\includegraphics[page=1]{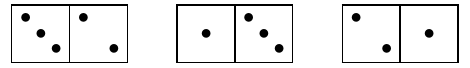}
\end{center}
vediamo le tessere $(3,2)$, $(1,3)$ e $(2,1)$.
Immaginiamo di voler tassellare con queste tessere una striscia orizzontale di altezza unitaria, in modo che quadrati adiacenti siano decorati con lo stesso numero di pallini.
Siccome l'altezza della striscia (e dei quadrati) è irrilevante, il problema è equivalente ad uno di tassellazione della retta reale.

Un modo per ricoprire la striscia $\R\times [0,1]$ con questo insieme di tessere è ripetere all'infinito la successione di tre tasselli seguente:
\begin{center}
\includegraphics[page=2]{domino.pdf}
\end{center}
Tale tassellazione è, per costruzione, periodica. Qualunque traslazione orizzontale di sei quadratini la lascia invariata.

Indichiamo con $V:=\{1,\ldots,n\}$ ($n\geq 1$) l'insieme usato per etichettare i quadrati (ogni elemento di $V$ corrisponde ad un numero di pallini).
Dato un insieme qualsiasi di tessere del domino, ovvero un sottoinsieme $E\subseteq V\times V$, ci possiamo chiedere se esiste una tassellazione
della striscia $\R\times [0,1]$ con protoinsieme $E$, ed in particolare se ne esiste una periodica oppure se ogni tassellazione con protoinsieme $E$ è aperiodica.
Per rispondere a queste domande possiamo tradurre il problema in uno di geometria combinatoria.

Ai fini della nostra discussione, definiamo un \emph{grafo} come una coppia $(V,E)$ in cui $V$ è un insieme finito (i \emph{vertici}) ed $E$ un sottoinsieme di $V\times V$ (gli \emph{archi}). 
Possiamo rappresentare graficamente un grafo disegnando un puntino $v$ per ogni vertice ed una freccia dal vertice $v_1$ al vertice $v_2$ se $(v_1,v_2)\in E$. 
Un \emph{cammino} (orientato), detto anche \emph{passeggiata}, in un grafo $(V,E)$ è una successione di vertici
$(v_0,v_1,\ldots,v_k)$ tali che
\begin{equation}\label{eq:cammino}
(v_{i-1},v_i)\in E
\end{equation}
per ogni $0<i\leq k$. Un \emph{cammino infinito} è una funzione $v:\N\to V$ soddisfacente la medesima condizione \eqref{eq:cammino}, per ogni $i\geq 1$. Un \emph{cammino doppiamente infinito} è una funzione $v:\Z\to V$ soddisfacente \eqref{eq:cammino}, per ogni $i\in\Z$.

Nel linguaggio della teoria dei grafi, una collezione di tessere del domino non è altro che l'insieme degli archi di un grafo, e una tassellazione della striscia $\R\times [0,1]$ non è altro che un cammino doppiamente infinito nel grafo. Ad esempio, l'insieme di quattro tessere
\begin{center}
\includegraphics[page=3]{domino.pdf}
\end{center}
corrisponde al grafo:
\begin{center}
\begin{tikzpicture}[
inner sep=3pt,
font=\footnotesize,
>=stealth,
node distance=2cm,
main node/.style={circle,inner sep=1.5pt,fill=black},
freccia/.style={->,shorten >=2pt,shorten <=2pt,semithick},
ciclo/.style={out=50, in=130, loop, distance=2cm, ->,semithick}
]

\clip (-0.5,-0.5) rectangle (4.4,1.4);

\node[main node] (1) {};
\node (2) [main node,right of=1] {};
\node (3) [main node,right of=2] {};

\filldraw (1) circle (0.06) node[below left] {$1$};
\filldraw (2) circle (0.06) node[below left] {$2$};
\filldraw (3) circle (0.06) node[below=2pt] {$3$};

\path[freccia] (2) edge[ciclo] (2);

\path[freccia]
	(1) edge (2)
	(2) edge[bend left=20] (3)
	(3) edge[bend left=20] (2);

\end{tikzpicture}\smallskip
\end{center}

Non è difficile dimostrare che in un grafo
se esiste un cammino infinito, allora esiste un circuito (un cammino finito in cui il primo ed ultimo vertice sono uguali), e quindi esiste un cammino doppiamente infinito e periodico (ottenuto ripetendo all'infinito il circuito).
Qualunque sia l'insieme $E$ di tessere considerato, si possono quindi verificare solo due casi:
\begin{enumerate}
\item non esiste tassellazione di $\R\times [0,1]$ con le tessere dell'insieme di $E$,
\item esistono tassellazioni di $\R\times [0,1]$ con le tessere di $E$, ed in particolare ne esistono periodiche.
\end{enumerate}
Se oltre alle traslazioni permettiamo anche che le tessere del domino siano ruotate, il problema diventa banale. Per ogni $(i,j)\in E$ si avrà anche $(j,i)\in E$, e se $E$ è non vuoto possiamo costruire un cammino doppiamente infinito andando avanti e indietro fra i due vertici di un qualsiasi arco.

\section{I tasselli di Wang}\label{sec:Wang}

Le tassellazioni di Wang \cite{Wan61} sono una versione in due dimensioni del problema del domino.
Un \emph{tassello di Wang} è un quadrato di lato unitario, più precisamente una cella del reticolo $\Z^2$, con lati colorati.
Per rendere più chiare le figure,
si usa in effetti colorare non i lati, ma piuttosto i quattro triangoli in cui il tassello viene suddiviso dalle diagonali del quadrato.
Ad esempio, un quadrato con i lati nord e sud verdi, est rosso e ovest celeste, è rappresentato come segue: 
\begin{center}
\begin{tikzpicture}

\wang{0,0}{cyan!50}{green!50}{red!50}{green!50}{}{}{}{};

\end{tikzpicture}
\end{center}
Associando a ciascun colore una etichetta presa da un insieme $C$ fissato,
un tassello di Wang $t$ può essere pensato come una quaterna $(t_n,t_s,t_o,t_e)$ di elementi di $C$ ($n/s/o/e$ sta per nord/sud/ovest/est, rispettivamente).
Come nel caso del domino, i tasselli di Wang possono essere traslati \emph{ma non ruotati}. Vogliamo ricoprire il piano posizionando tasselli di Wang nelle celle di $\Z^2$, in modo che quadrati adiacenti condividano esclusivamente lati con lo stesso colore.

Più formalmente, dato un insieme $P\subseteq C^4$ di tasselli di Wang, con colori presi da un insieme finito $C$, una \emph{tassellazione di Wang} con protoinsieme $P$ è una funzione $t:\Z^2\to P$ tale che
\begin{equation}\label{eq:wang}
\begin{split}
t(i,j)_e &=t(i+1,j)_o \;, \\
t(i,j)_n &=t(i,j+1)_s \;,
\end{split}
\end{equation}
per ogni $(i,j)\in\Z^2$.
Una tassellazione di Wang si dice \emph{periodica} se esiste $(a,b)\in\Z^2\smallsetminus\{(0,0)\}$ tale che $t(i+a,j+b)=t(i,j)$ per ogni $(i,j)\in\Z^2$.
Una tassellazione di Wang si dice \emph{bi-periodica} se esistono $a,b\in\Z\smallsetminus\{0\}$ tale che $t(i+a,j)=t(i,j)=t(i,j+b)$ per ogni $(i,j)\in\Z^2$.

Un esempio è il protoinsieme formato dai tre tasselli:
\begin{center}
\begin{tikzpicture}
\wangA{0,0}
\wangB{3,0}
\wangC{6,0}
\end{tikzpicture}
\end{center}
con il quale è possibile costruire sia tassellazioni bi-periodiche che aperiodiche.
Una tassellazione bi-periodica è ad esempio quella costante, che associa ad ogni cella di $\Z^2$ il primo tassello. Una porzione di questa tassellazione ha il seguente aspetto:
\begin{center}
\begin{tikzpicture}
\foreach \x in {0,...,5} {
	\foreach \y in {0,...,2} {
		\wangA{\x,\y}
		}
	}
\end{tikzpicture}
\end{center}
Per ottenere una tassellazione aperiodica possiamo usare un trucco analogo a quello utilizzato nel caso della tassellazione a nido d'ape. Partiamo dalla tassellazione costante precedente e sostituiamo una qualsiasi coppia di tessere adiacenti come segue:
\begin{center}
\begin{tikzpicture}
\foreach \x in {0,...,5} {
	\foreach \y in {0,...,2} {
		\wangA{\x,\y}
		}
	}
\wangB{2,1}
\wangC{3,1}
\end{tikzpicture}
\end{center}
rompendo in questo modo la simmetria traslazionale.

\begin{thm}\label{thm:wang}
Se esiste una tassellazione di Wang periodica con protoinsieme $P$, allora esiste anche una tassellazione bi-periodica con lo stesso protoinsieme.
\end{thm}

\begin{proof}
Assumiamo che $t:\Z^2\to P$ sia periodica di periodo $(a,b)\in\Z^2\smallsetminus\{(0,0)\}$. Senza perdere generalità, possiamo assumere che $a,b\geq 0$.
Per semplicità di notazione, iniziamo con il caso $a=0$ e $b>0$.

Notiamo che ogni striscia verticale della tassellazione è periodica con lo stesso periodo $b$.
Siccome l'insieme $C$ di colori è finito, anche l'insieme $P$ di prototasselli è finito ($P\subseteq C^4$), e con essi è possibile costruire solo un numero finito di strisce verticali periodiche di periodo $b$, siano esse $S_1,S_2,\ldots,S_n$ (il loro numero è al più $|P|^{b}$). Concretamente, ciascuna striscia $S_k$ è una funzione $t(i_k,\,.\,):\Z\to P$ soddisfacente la seconda condizione in \eqref{eq:wang}.

Associamo ad ogni striscia un vertice di un grafo, e ad ogni coppia $(S_h,S_k)$ un'arco se e solo se attaccando $S_k$ alla destra di $S_h$ le regole di corrispondenza vengono rispettate ($t(i_h,j)_e=t(i_k,j)_o$). Ad ogni cammino doppiamente infinito in questo grafo corrisponde allora una tassellazione di Wang, e ad ogni cammino periodico corrisponde una tassellazione di Wang bi-periodica.
Siccome esiste almeno un cammino doppiamente infinito, corrispondente alla tassellazione $t$ di partenza, deve esisterne anche uno periodico, il che conclude la dimostrazione quando $a=0$.

Nel caso in cui il periodo sia dato da un vettore $(a,b)\in\Z^2$ con $a>0$ e $b=0$, si può ripetere la stessa dimostrazione con strisce orizzontali invece di verticali. Nel caso in cui sia $a$ che $b$ sono positivi, invece di strisce si possono considerare delle scale come in Figura~\ref{fig:scala}.

\begin{figure}[H]
\begin{tikzpicture}[scale=0.5,font=\tiny]

\foreach \k in {0,1,2} \fill[gray!20] (4*\k,4*\k) rectangle (4*\k+4,4*\k+1);

\draw (-0.5,-0.5) grid (12.5,9.5);

\node at (0.5,0.5) {$1$};
\node at (1.5,0.5) {$2$};
\node at (2.55,0.5) {$\cdots$};
\node at (3.5,0.5) {$a$};
\node at (4.5,0.5) {$1$};
\node at (4.5,1.5) {$2$};
\node at (4.5,2.7) {$\vdots$};
\node at (4.5,3.5) {$b$};

\node at (4.5,4.5) {$1$};
\node at (5.5,4.5) {$2$};
\node at (6.55,4.5) {$\cdots$};
\node at (7.5,4.5) {$a$};
\node at (8.5,4.5) {$1$};
\node at (8.5,5.5) {$2$};
\node at (8.5,6.7) {$\vdots$};
\node at (8.5,7.5) {$b$};

\node at (8.5,8.5) {$1$};
\node at (9.5,8.5) {$2$};
\node at (10.55,8.5) {$\cdots$};
\node at (11.5,8.5) {$a$};

\end{tikzpicture}
\caption{}\label{fig:scala}
\end{figure}

La dimostrazione resta pressoché invariata: ciascuna scala è periodica;
esiste al più un numero finito di scale di questo tipo (al massimo $|P|^{ab}$);
trasformando le scale nei vertici di un grafo, ogni cammino doppiamente infinito corrisponde ad una tassellazione di Wang, ed ogni cammino periodico corrisponde ad una tassellazione bi-periodica.
\end{proof}

Dato un insieme $P$ di prototasselli di Wang, il principale problema è naturalmente capire se esiste una tassellazione di Wang con protoinsieme $P$.
Il \emph{teorema di estensione} garantisce l'esistenza di tassellazioni di Wang con protoinsieme $P$ se (e solo se) è possibile tassellare con lo stesso protoinsieme quadrati arbitrariamente grandi (si veda ad esempio \cite[Sect.~2.5.3]{DanPen} per la dimostrazione nel contesto generale delle tassellazioni del piano). Tale affermazione non è affatto ovvia, e la dimostrazione del teorema non è costruttiva (usa il fatto che i tasselli sono compatti) e non dà una indicazione concreta su come trovare esplicitamente una tassellazione del piano.

Sarebbe bello poter scrivere un programma al computer che prendesse come input un insieme $P$ di prototasselli e rispondesse in un tempo finito alla domanda ``\textit{esiste una tassellazione di Wang con protoinsieme $P$?}'' 
Il problema di esistenza di un tale programma è stato posto da H.~Wang in \cite{Wan61},
e questo tipo di domande è in effetti proprio la motivazione originale per la costruzione di questi tasselli.

Un algoritmo potrebbe essere il seguente. Per ogni $n\in\{1,2,3,\ldots\}$ consideriamo il quadrato $[0,n]^2$. Siccome per ogni $P$ esiste solo un numero finito di funzioni $t:[0,n]^2\to P$, possiamo controllarle una ad una e vedere se ne esiste almeno una soddisfacente \eqref{eq:wang} nel suo dominio di definizione. Il programma si arresta se per un qualche $n$ scopriamo che nessuna funzione $t:[0,n]^2\to P$ soddisfa \eqref{eq:wang}, e risponde (come conseguenza del teorema di estensione) in modo negativo:
\begin{itemize}
\item[]
\small\texttt{non esiste alcuna tassellazione di Wang \\ con protoinsieme $P$.}
\end{itemize}
Naturalmente, se esiste una tassellazione di Wang, l'algoritmo continuerà a testare funzioni per $n$ sempre più grande senza mai arrestarsi.

Non è difficile ideare anche un algoritmo per testare l'esistenza di tassellazioni periodiche.
Ordiniamo coppie $(m,n)$ di interi positivi ad esempio in modo lessicografico.
Per ciascuna coppia, consideriamo tutte le funzioni $t:\{0,\ldots,m\}\times\{0,\ldots,n\}$ soddisfacenti \eqref{eq:wang} nel loro dominio di definizione. Se il programma ne trova una che in aggiunta soddisfi $t(m,j)_e=t(0,j)_o$ e $t(i,n)_n=t(i,0)_s$
per ogni $0\leq i\leq m$ e $0\leq j\leq n$, si arresta e risponde:
\begin{itemize}
\item[]
\small\texttt{esiste una tassellazione di Wang bi-periodica \\ con protoinsieme $P$.}
\end{itemize}
Tale tassellazione bi-periodica $\overline{t}$ si può costruire a partire dalla funzione $t$ trovata ponendo $\overline{t}(i+hm,j+kn):=t(i,j)$ per ogni $0\leq i\leq m$, $0\leq j\leq n$ e per ogni $h,k\in\Z$.
Siccome l'esistenza di una tassellazione periodica implica quella di una tassellazione bi-periodica, l'unico caso in cui il programma continua all'infinito è se non esistono tassellazioni periodiche con protoinsieme $P$.

Combinando questi due algoritmi si può ottenere un programma che, prendendo come input un insieme $P$ di tasselli di Wang, risponda alla nostra domanda in un tempo finito se non esiste alcuna tassellazione di Wang con protoinsieme $P$, o se ne esiste almeno una periodica.
Abbiamo quasi risolto il problema che ci siamo posti in partenza: l'unico caso in cui il nostro programma non si arresta è se gli diamo come input un protoinsieme $P$ per il quale esistono tassellazioni di Wang e sono tutte aperiodiche. Diremo che tale insieme è \emph{aperiodico}.

L'esistenza di un protoinsieme aperiodico è stata provata da R.~Berger, uno studente di H.~Wang, nella sua tesi di dottorato, assieme alla dimostrazione che il problema di Wang è indecidibile \cite{Ber66}. Il primo protoinsieme aperiodico scoperto da Berger era formato da oltre 20 mila tasselli, ma nella tesi di dottorato il numero di tasselli necessari si era già ridotto a 104. Recentemente, E.~Jeandel ed M.~Rao in \cite{JR21} hanno costruito un protoinsieme aperiodico formato da 11 tasselli di Wang e usando solamente 5 colori (4 se non contiamo il bianco), e dimostrato che tale numero è minimale (non esiste protoinsieme aperiodico di tasselli di Wang con meno di 11 tasselli o meno di 5 colori).
La dimostrazione in \cite{JR21} fa un uso estensivo del computer (come il celebre teorema dei quattro colori).

Qui proponiamo l'esempio scoperto da J.~Kari \cite{Kar96}, la cui dimostrazione è estremamente elegante e non richiede l'uso di un computer.
Il protoinsieme di Kari è formato dai seguenti 14 tasselli:
\begin{gather}
\includegraphics[page=1,scale=0.75,valign=c]{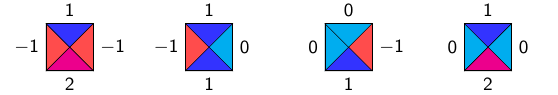} \label{eq:kari1} \\
\includegraphics[page=2,scale=0.75,valign=c]{wang.pdf} \label{eq:kari2} \\
\includegraphics[page=3,scale=0.75,valign=c]{wang.pdf} \label{eq:kari3}
\end{gather}
In questa figura vediamo 8 colori, ciascuno associato ad un numero. Celeste ed arancione sono indicati con $0$ e $0'$: entrambi valgono $0$ ai fini aritmetici, ma l'apice ci ricorda che sono colori differenti e non possono essere adiacenti. Siccome però i tasselli non possono essere ruotati ma solo traslati, i colori verticali si possono scegliere indipendentemente da quelli orizzontali ed è possibile ricolorare i tasselli usando solo 6 colori.

Vogliamo dimostrare che il protoinsieme di Kari è aperiodico. Dobbiamo provare che
esiste almeno una tassellazione di Wang con tale protoinsieme (Prop.~\ref{prop:kari2}),
e che ogni tassellazione di Wang con tale protoinsieme è aperiodica (Prop.~\ref{prop:kari1}).

Per iniziare, cerchiamo di capire da dove vengono questi tasselli. 
Dato un numero reale positivo $\alpha$ ed un numero razionale positivo $r$, indichiamo con $A(r,\alpha)$ e $B(\alpha)$ le due successioni con elementi:
\begin{align*}
A_k(r,\alpha) &:=r\lfloor k\alpha \rfloor- \lfloor rk\alpha \rfloor \;,
\\
B_k(\alpha) &:=\lfloor k\alpha \rfloor- \lfloor (k-1)\alpha \rfloor \;,
\end{align*}
per ogni $k\in\Z$. In queste formule, $\lfloor x\rfloor:=\max\{n\in\Z\mid n\leq x\}$ è la parte intera di $x$.

Non è difficile mostrare che entrambe le successioni assumono solo un numero finito di valori.
Siccome $k\alpha-1<\lfloor k\alpha \rfloor\leq k\alpha$, si ha
\[
\alpha-1<B_k(\alpha)<\alpha+1 .
\]
Ma essendo $B_k(\alpha)$ intero, deve essere $B_k(\alpha)\in \big\{ \lfloor \alpha \rfloor,\lfloor \alpha \rfloor+1\big\}$. In particolare, $B(\alpha)\in\{0,1\}^{\Z}$ se $\alpha\in [0,1]$ e $B(\alpha)\in\{0,1,2\}^{\Z}$ se $\alpha\in [0,2]$ ($B_k(\alpha)=\alpha$ se $\alpha$ è intero).

Posto $x:=k\alpha$, siccome
\[
r\lfloor x \rfloor-1\leq rx-1<\lfloor rx \rfloor\leq rx< r(\lfloor x \rfloor+1)
\]
si ha
\[
-r<A_k(r,\alpha)<1 .
\]
Scrivendo $r=p/q$ con $p,q$ interi positivi, siccome $A_k(r,\alpha)$ è un multiplo intero di $1/q$, si ha
\[
A_k(r,\alpha)\in\left\{
-r+\frac{1}{q},-r+\frac{2}{q},\ldots,1-\frac{1}{q}
\right\} .
\]
Definiamo ora due insiemi di tasselli di Wang, $T_2$ e $T_{2/3}$, come segue.

\begin{df}
Per $r\in\{2,2/3\}$, $T_r$ è l'insieme di tutti i tasselli della forma
\begin{equation}\label{eq:task}
\begin{tikzpicture}[thick,font=\footnotesize,baseline=(current bounding box.center)]

\draw[very thin] (0,0) -- (1,1) (0,1) -- (1,0);
\draw (0,0) rectangle (1,1);

\node[above] at (0.5,1) {$B_k(\alpha)$};
\node[below] at (0.5,0) {$B_k(r\alpha)$};
\node[left] at (0,0.5) {$A_{k-1}(r,\alpha)$};
\node[right] at (1,0.5) {$A_k(r,\alpha)$};

\end{tikzpicture}
\end{equation}
con $k\in\Z$ ed $\alpha$ che varia in $[1/2,1]$ se $r=2$, ed in $(1,2]$ se $r=2/3$.
\end{df}

Notiamo che ogni tassello dell'insieme $T_r$ è del tipo
\begin{center}
\begin{tikzpicture}[thick,font=\footnotesize]

\draw[very thin] (0,0) -- (1,1) (0,1) -- (1,0);
\draw (0,0) rectangle (1,1);

\node[above] at (0.5,1) {$a$};
\node[below] at (0.5,0) {$c$};
\node[left] at (0,0.5) {$b$};
\node[right] at (1,0.5) {$d$};

\end{tikzpicture}
\end{center}
con $a,b,c,d$ interi soddisfacenti
\begin{equation}\label{eq:divisione}
ra+b=c+d .
\end{equation}
Possiamo allora elencare tutti i tasselli di questi due insiemi.

Un tassello in $T_2$ ha $b,d\in\{-1,0\}$, $a\in\{0,1\}$ (siccome $\alpha\in [0,1]$)
e $c\in\{1,2\}$ (siccome $2\alpha\in [1,2]$). Esistono solo quattro tasselli di questo tipo,
soddisfacenti \eqref{eq:divisione} con $r=2$, e sono i quattro tasselli in \eqref{eq:kari1}.

Un tassello in $T_{2/3}$ ha $b,d\in\{-\frac{1}{3},0,\frac{1}{3},\frac{2}{3}\}$, $a\in\{1,2\}$ (siccome $\alpha\in [1,2]$) e $c\in\{0,1,2\}$ (siccome $\frac{2}{3}\alpha\in [0,1]$). Esistono 
dieci tasselli di questo tipo soddisfacenti \eqref{eq:divisione} con $r=2/3$, e sono esattamente tasselli in \eqref{eq:kari2} ed in \eqref{eq:kari3}.

La distinzione fra $0$ e $0'$ --- la colorazione diversa dei lati corrispondenti --- fa sì che ogni riga orizzontale di una tassellazione di Wang con questi tasselli sia formata solo da elementi di $T_2$, oppure solo da elementi di $T_{2/3}$.

\begin{prop}\label{prop:kari2}
Esiste una tassellazione di Wang con tasselli (\ref{eq:kari1}-\ref{eq:kari3}).
\end{prop}

\begin{proof}
Sia $f:[\frac{1}{2},2]\to [\frac{1}{2},2]$ la funzione
\[
f(x) :=\begin{cases}
2x & \text{se }x\in[\frac{1}{2},1] ,\\
\frac{2}{3}x & \text{se }x\in (1,2] .
\end{cases}
\]
Per $\alpha\in [\frac{1}{2},2]$, siccome i numeri sui lati verticali del tassello \eqref{eq:task} sono consecutivi, possiamo costruire una riga orizzontale con questi tasselli che sopra ha gli elementi della successione $B(\alpha)$ e sotto gli elementi della successione $B(r\alpha)$ dove $r=f(\alpha)$. 
Indichiamo tale riga con $R(\alpha)$.

Sia $g:[\frac{2}{3},2]\to [\frac{2}{3},2]$ la funzione
\[
g(x) :=\begin{cases}
\frac{1}{2}x & \text{se }x\in[\frac{2}{3},\frac{4}{3}] ,\\
\frac{3}{2}x & \text{se }x\in (\frac{4}{3},2] .
\end{cases} \\
\]
Per ogni $\alpha\in [\frac{2}{3},2]$, sopra alla riga $R(\alpha)$ possiamo attaccare la riga $R(g(\alpha))$, e sotto possiamo attaccare la riga $R(f(\alpha))$.

Scelto un qualsiasi numero $\alpha_0\in[\frac{2}{3},2]$. Definiamo ricorsivamente
$\alpha_n=g(\alpha_{n-1})$ e $\alpha_{-n}=f(\alpha_{-n+1})$ per $n$ intero positivo.
Partendo dalla riga $R(\alpha_0)$, possiamo attaccarle sopra $R(\alpha_1)$, $R(\alpha_2)$, etc.\
ed attaccarle sotto $R(\alpha_{-1})$, $R(\alpha_{-2})$, etc.\ e ricoprire in questo modo il piano ottenendo una tassellazione di Wang.
\end{proof}

\begin{prop}\label{prop:kari1}
Ogni tassellazione di Wang con tasselli (\ref{eq:kari1}-\ref{eq:kari3}) è aperiodica.
\end{prop}

\begin{proof}
Dimostriamo la proposizione per assurdo. Per il Teorema \ref{thm:wang}, senza perdere generalità assumiamo che esista una tassellazione di Wang $t:\Z^2\to T_2\cup T_{2/3}$ bi-periodica con periodi $a,b>0$. Per $j\in\Z$ indichiamo con $q_j$ la somma dei colori sul lato nord dei quadrati
$t(1,j),\ldots,t(a,j)$. Quindi:
\[
q_j:=t(1,j)_n+t(2,j)_n+\ldots+t(a,j)_n .
\]
Siccome $t(i+a,j)=t(i,j)$, la somma dei colori sul lato ovest è uguale alla somma dei colori sul lato est.
Siccome i tasselli sono tutti dell'insieme $T_2$ o $T_{2/3}$, da \eqref{eq:divisione}
ricaviamo
\[
r_iq_i=q_{i+1} ,
\]
dove $r_i=2$ se i tasselli sono presi dall'insieme $T_2$, e $r_i=2/3$ se sono presi dall'insieme $T_{2/3}$. Siccome la tassellazione è periodica in direzione verticale di periodo $b$, si ha
\[
q_0=q_b=r_{b-1}\cdots r_1r_0\cdot q_0 .
\]
Abbiamo raggiunto un assurdo.
Siccome due tasselli con $0$ sul lato nord non possono essere adiacenti, si ha $q_0\neq 0$.
Ma $r_{b-1}\cdots r_1r_0$ non può essere uguale ad $1$, perché i fattori sono $2$ e $2/3$.
\end{proof}

\begin{cor}
L'insieme di tasselli (\ref{eq:kari1}-\ref{eq:kari3}) è aperiodico.
\end{cor}

L'interesse in logica matematica ed in informatica per le tassellazioni di Wang viene dal fatto che è possibile trovare un protoinsieme di Wang che simuli il comportamento di qualsiasi \emph{macchina di Turing} \cite{Wan75}. Vari esempi concreti si possono trovare in \cite{GS87}.

\section{Le tassellazioni di Penrose}\label{sec:pen1}

Una delle domande poste da Keplero nel suo \textit{Harmonices mundi} era se fosse possibile tassellare il piano usando solo tasselli con simmetria 5-fold. Sappiamo che non è possibile farlo usando un singolo prototassello \cite{DGS82}, ad esempio un pentagono regolare o un pentagramma. Ma a parte questo, il problema è tutt'ora aperto.
Nel 1973, cercando una soluzione al problema posto da Keplero, Sir.~Roger Penrose (Premio Nobel per la fisica nel 2020) scoprì un protoinsieme aperiodico formato da soli sei prototasselli: tre pentagoni regolari (con diverse regole di corrispondenza), un pentagramma, un rombo ed una ``barchetta'' \cite{Pen74}. Un anno dopo, ridusse il numero di prototasselli a due soli, i suoi celebri dardo ed aquilone illustrati in \eqref{eq:dardoaq}.
Questi due protoinsiemi, ed un terzo protoinsieme formato da due rombi, sono descritti in \cite{Pen78}.

Noi inizieremo con dei tasselli un po' diversi, introdotti da Robinson per studiare le tassellazioni di Penrose \cite{Rob75}.
Questi tasselli si ottengono sezionando un pentagono regolare di lato unitario, come in figura:
\begin{center}
\begin{tikzpicture}[semithick,font=\footnotesize,inner sep=0.5pt]

\coordinate (v0) at (0,0);
\coordinate (v1) at ($(v0)+(36:1.5)$);
\coordinate (v2) at ($(v1)+(108:1.5)$);
\coordinate (v3) at ($(v2)+(180:1.5)$);
\coordinate (v4) at ($(v3)+(252:1.5)$);

\draw (v4) -- (v1) -- (v2) -- (v3) -- (v4) -- (v2);
\draw[dashed] (v1) -- (v0) -- (v4);

\foreach \k in {0,...,4} \fill (v\k) circle (0.03);

\node[above] at (-0.75,0.9) {1};
\node[above] at (1,0.9) {2};
\node at (0.65,2) {2};
\node at (-0.95,1.3) {1};
\node[above] at (0.25,2) {1};
\node[above] at (-0.65,1.98) {3};

\end{tikzpicture}
\end{center}
I numeri in figura indicano l'ampiezza di ciascun angolo, in multipli di $\pi/10$. Otteniamo due triangoli
\begin{center}
\begin{tikzpicture}[font=\small,baseline=(current bounding box.center)]

\coordinate (v1) at (0,0);
\coordinate (v2) at ($(v1)+(1.5,0)$);
\coordinate (v3) at ($(v1)+(72:1.5*\golden)$);

\draw[semithick] (v1) -- (v2) -- (v3) -- cycle;

\foreach \k in {1,2,3} \fill (v\k) circle (0.03);

\begin{scope}[font=\footnotesize,inner sep=0pt]
\node[shift={(36:0.3)}] at (v1) {$2$};
\node[shift={(144:0.3)}] at (v2) {$2$};
\node[shift={(-90:0.5)}] at (v3) {$1$};

\draw [gray,decorate,decoration = {brace,mirror,amplitude=7pt,raise=5pt}] (v1) -- (v2) node[midway,below=16pt] {$1$};

\draw [gray,decorate,decoration = {brace,amplitude=7pt,raise=5pt}] (v1) -- (v3) node[midway,left=15pt,yshift=5pt] {$\varphi$};

\end{scope}

\end{tikzpicture}\qquad\qquad
\begin{tikzpicture}[font=\small,baseline=(current bounding box.center)]

\coordinate (v4) at (0,0);
\coordinate (v5) at ($(v4)+(1.5*\golden,0)$);
\coordinate (v6) at ($(v4)+(36:1.5)$);

\draw[semithick] (v4) -- (v5) -- (v6) -- cycle;

\foreach \k in {4,5,6} \fill (v\k) circle (0.03);

\begin{scope}[font=\footnotesize,inner sep=0pt]
\node[shift={(18:0.55)}] at (v4) {$1$};
\node[shift={(162:0.55)}] at (v5) {$1$};
\node[shift={(-90:0.2)}] at (v6) {$3$};

\draw [gray,decorate,decoration = {brace,mirror,amplitude=7pt,raise=5pt}] (v4) -- (v5) node[midway,below=16pt] {$\varphi$};

\draw [gray,decorate,decoration = {brace,amplitude=7pt,raise=5pt}] (v4) -- (v6) node[midway,above=12pt,xshift=-10pt] {$1$};

\end{scope}

\end{tikzpicture}
\end{center}

\pagebreak

\end{multicols}

\begin{figure}[!t]

\vspace*{3cm}

\includegraphics[width=0.8\textwidth]{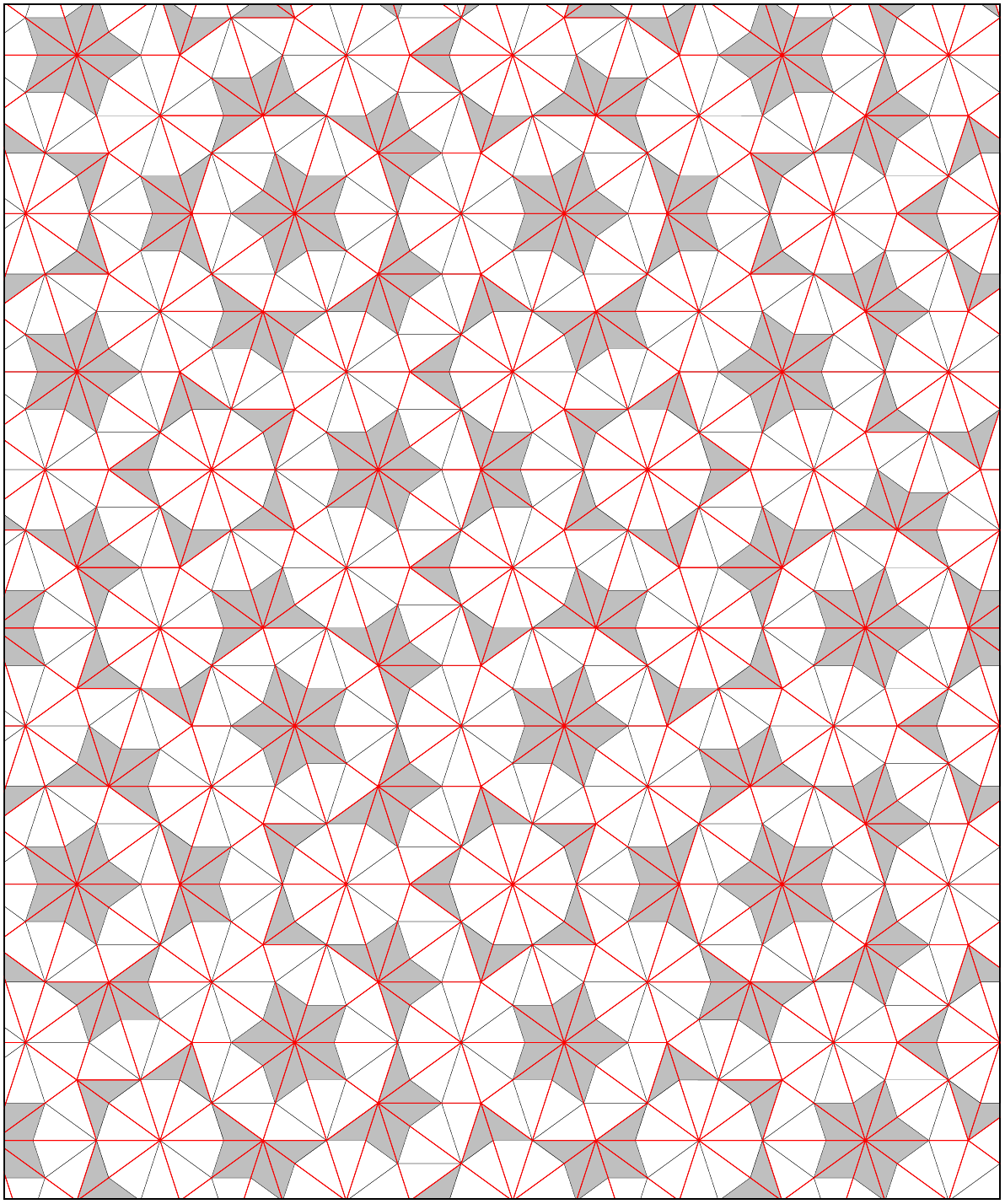}

\caption{Una tassellazione di Penrose con triangoli}\label{fig:penrose}
\end{figure}

\begin{multicols}{2}

~\pagebreak

\noindent
I lati corti hanno lunghezza unitaria, mentre i lunghi hanno lunghezza pari alla \emph{sezione aurea}
\[
\varphi:=\frac{1+\sqrt{5}}{2} .
\]
Un triangolo isoscele in cui il rapporto dei lati è la sezione aurea è detto \emph{triangolo aureo}. Ogni triangolo aureo è simile ad uno dei due in figura. Se l'angolo opposto alla base è acuto, parleremo di \emph{acutaureo}; se è ottuso, parleremo di \emph{ottusaureo}.

Per ottenere delle tassellazioni interessanti ci mancano delle regole di corrispondenza. Decoriamo i lati con delle frecce come segue:
\begin{equation}\label{eq:Rtriangles}
\begin{tikzpicture}[baseline=(current bounding box.center),font=\small,>=Straight Barb]

\coordinate (a) at (0,0);
\coordinate (b1) at (252:1.5*\golden);
\coordinate (b2) at (-72:1.5*\golden);

\fill[TA] (a) -- (b1) -- (b2) -- cycle;

\begin{scope}[decoration={
    markings,
    mark=at position 0.5 with {\arrow[scale=1.1,xshift=1.5pt]{>}}}
    ]
\draw[postaction={decorate}] (b1) -- (a);
\end{scope}

\begin{scope}[decoration={
    markings,
    mark=at position 0.5 with {\arrow[scale=1.1,xshift=3pt]{>>}}}
    ]
\draw[postaction={decorate}] (a) -- (b2);
\draw[postaction={decorate}] (b2) -- (b1);
\end{scope}


\end{tikzpicture} \qquad\qquad
\begin{tikzpicture}[baseline={(current bounding box.center)},font=\small,>=Straight Barb]

\coordinate (a) at (0,0);
\coordinate (b1) at (216:1.5);
\coordinate (b2) at (-36:1.5);

\fill[SA] (a) -- (b1) -- (b2) -- cycle;

\begin{scope}[decoration={
    markings,
    mark=at position 0.5 with {\arrow[scale=1.1,xshift=1.5pt]{>}}}
    ]
\draw[postaction={decorate}] (a) --  (b1);
\end{scope}

\begin{scope}[decoration={
    markings,
    mark=at position 0.5 with {\arrow[scale=1.1,xshift=3pt]{>>}}}
    ]
\draw[postaction={decorate}] (b2) -- (b1);
\draw[postaction={decorate}] (a) -- (b2);\end{scope}


\end{tikzpicture}
\end{equation}

Ci interessa studiare tassellazioni del piano con protoinsieme \eqref{eq:Rtriangles}, edge-to-edge, e con la regola che due lati possono essere uniti solo se hanno la stessa freccia (singola o doppia) e che punta nella stessa direzione. Questi due triangoli possono essere traslati, \emph{ruotati} (al contrario di quanto accadeva con le tassellazioni di Wang) e anche \emph{riflessi}. Abbiamo quindi un protoinsieme di due o quattro prototasselli, a seconda che contiamo le immagini riflesse dei triangoli come prototasselli distinti.

Ora che abbiamo questo protoinsieme, possiamo costruire le prime tassellazioni. Una porzione di tassellazione del piano è in Figura~\ref{fig:penrose} (le frecce sui lati dei triangoli sono omesse per non appesantire la figura). Un osservatore attento noterà che, sebbene ad occhio non si veda alcuna simmetria in questa tassellazione, ci sono dei motivi di base ripetuti nella figura. Notiamo vari tipi di ``stelle'', delle ``ruote'' formate da dieci acutaurei disposti a raggiera, etc.

Lo tassellazioni di Penrose possono essere ottenute in tanti modi equivalenti, usando diversi tipi di prototasselli. Per passare da una descrizione all'altra, ed in generale per studiare le proprietà di queste tassellazioni, utilizziamo alcune operazioni di ``taglio e cucito'' che descriveremo ora nel seguito in dettaglio.

Per cominciare, notiamo che i triangoli in \eqref{eq:Rtriangles} possono essere combinati (rispettando le regole di corrispondenza) a formare l'aquilone ed il dardo in \eqref{eq:dardoaq}:
\begin{equation}\label{eq:KD}
\begin{tikzpicture}[baseline=(current bounding box.center),font=\small,>=Straight Barb]

\clip (-1.5,-0.2) rectangle (1.5,2.3);

\coordinate (a4) at (0,0);
\coordinate (a2) at (162:\raggio);
\coordinate (a3) at (18:\raggio);
\path[name path=PB] (a2)-- ++(54:3);
\path[name path=PA] (a3)-- ++(126:3);
\path[name intersections={of=PA and PB,by=a1}];

\fill[GTB] (a1) -- (a2) -- (a4) -- cycle;
\fill[GTA] (a1) -- (a3) -- (a4) -- cycle;

\begin{scope}[decoration={
    markings,
    mark=at position 0.5 with {\arrow[scale=1.1,xshift=1.5pt]{>}}}
    ]
\draw[dashed,postaction={decorate}] (a4) -- (a1);
\end{scope}

\begin{scope}[decoration={
    markings,
    mark=at position 0.5 with {\arrow[scale=1.1,xshift=3pt]{>>}}}
    ]
\draw[postaction={decorate}] (a1) -- (a2);
\draw[postaction={decorate}] (a1) -- (a3);
\draw[postaction={decorate}] (a2) -- (a4);
\draw[postaction={decorate}] (a3) -- (a4);
\end{scope}

\end{tikzpicture}
\hspace{1cm}
\begin{tikzpicture}[baseline=(current bounding box.center),font=\small,>=Straight Barb]

\coordinate (a1) at (0,0);
\coordinate (a3) at (-18:\raggio);
\coordinate (a2) at (198:\raggio);
\coordinate (a4) at (0,\raggio);

\fill[GSA] (a1) -- (a2) -- (a4) -- cycle;
\fill[GSB] (a1) -- (a3) -- (a4) -- cycle;

\begin{scope}[decoration={
    markings,
    mark=at position 0.5 with {\arrow[scale=1.1,xshift=1.5pt]{>}}}
    ]
\draw[dashed,postaction={decorate}] (a1) -- (a4);
\end{scope}

\begin{scope}[decoration={
    markings,
    mark=at position 0.5 with {\arrow[scale=1.1,xshift=3pt]{>>}}}
    ]
\draw[postaction={decorate}] (a2) -- (a4);
\draw[postaction={decorate}] (a3) -- (a4);
\draw[postaction={decorate}] (a1) -- (a2);
\draw[postaction={decorate}] (a1) -- (a3);
\end{scope}

\end{tikzpicture}
\end{equation}
Per questo motivo, il triangolo celeste in \eqref{eq:Rtriangles} è chiamato \emph{semiaquilone}, mentre quello rosso è detto \emph{semidardo}.
Nella figura \eqref{eq:KD} e in quelle seguenti, usiamo il colore bianco per i triangoli \eqref{eq:Rtriangles} ed il colore grigio per le loro immagini riflesse.

E' chiaro che ogni tassellazione con protoinsieme \eqref{eq:KD} si può trasformare in una con protoinsieme \eqref{eq:Rtriangles} dividendo aquilone e dardo a metà. Viceversa, studiando le regole di corrispondenza non è difficile realizzare che in ogni tassellazione con protoinsieme \eqref{eq:Rtriangles} gli acutaurei sono combinati a coppie a formare tanti aquiloni, e gli ottusaurei sono combinati a coppie a formare tanti dardi. Sì può quindi passare da una tassellazione con triangoli ad una con aquiloni e dardi rimuovendo tutti i lati con freccia singola.
Siccome ogni tassellazione del piano con triangoli si può trasformare in una con aquilone e dardo, e viceversa, diciamo che i due protoinsiemi sono \emph{mutuamente equivalenti}.

Riscalando l'ottusaureo di un fattore $\varphi$ otteniamo un altro protoinsieme, formato da un acutaureo con lati $1$  e $\varphi$, ed un ottusaureo con lati $\varphi$ e $\varphi^2=\varphi+1$. Questi triangoli possono essere combinati con le loro immagini riflesse a formare dei rombi:
\begin{equation}\label{eq:Krombi}
\begin{tikzpicture}[baseline=(current bounding box.center),font=\small,>=Straight Barb,rotate=90]

\coordinate (a4) at (0,0);
\coordinate (a2) at (162:1.618*\raggio);
\coordinate (a3) at (18:1.618*\raggio);
\coordinate (a1) at ($(a2)+(18:1.618*\raggio)$);

\fill[GTA] (a1) -- (a2) -- (a4) -- cycle;
\fill[GTB] (a1) -- (a3) -- (a4) -- cycle;

\begin{scope}[decoration={
    markings,
    mark=at position 0.5 with {\arrow[scale=1.1,xshift=1.5pt]{>}}}
    ]
\draw[postaction={decorate}] (a4) -- (a2);
\draw[postaction={decorate}] (a4) -- (a3);
\end{scope}

\begin{scope}[decoration={
    markings,
    mark=at position 0.5 with {\arrow[scale=1.1,xshift=3pt]{>>}}}
    ]
\draw[postaction={decorate}] (a2) -- (a1);
\draw[postaction={decorate}] (a3) -- (a1);
\draw[dashed,postaction={decorate}] (a1) -- (a4);
\end{scope}

\end{tikzpicture}
\hspace{1cm}
\begin{tikzpicture}[baseline=(current bounding box.center),font=\small,>=Straight Barb,rotate=90]

\coordinate (a1) at (0,0);
\coordinate (a2) at (-36:1.618*\raggio);
\coordinate (a3) at (36:1.618*\raggio);
\coordinate (a4) at ($(a2)+(36:1.618*\raggio)$);

\fill[GSA] (a1) -- (a2) -- (a4) -- cycle;
\fill[GSB] (a1) -- (a3) -- (a4) -- cycle;

\begin{scope}[decoration={
    markings,
    mark=at position 0.5 with {\arrow[scale=1.1,xshift=1.5pt]{>}}}
    ]
\draw[postaction={decorate}] (a2) -- (a4);
\draw[postaction={decorate}] (a3) -- (a4);
\end{scope}

\begin{scope}[decoration={
    markings,
    mark=at position 0.5 with {\arrow[scale=1.1,xshift=3pt]{>>}}}
    ]
\draw[dashed,postaction={decorate}] (a1) -- (a4);
\draw[postaction={decorate}] (a2) -- (a1);
\draw[postaction={decorate}] (a3) -- (a1);
\end{scope}

\end{tikzpicture}
\end{equation}
Per questo motivo, acutaureo e ottusaureo riscalato sono anche detti \emph{semi-rombi}. Ogni tassellazione del piano con rombi \eqref{eq:Krombi} si può convertire in una con semi-rombi dividendo i rombi a metà. Viceversa, in una tassellazione con semirombi la base di un acutaureo si può attaccare solo alla base della sua immagine riflessa (non ci sono altri lati di lunghezza $1$, e la base dell'ottusaureo si può attaccare solo alla base della sua immagine riflessa (non ci sono altri lati di lunghezza $\varphi+1$). Da ogni tassellazione del piano con semi-rombi ne possiamo quindi ottenere una con rombi rimuovendo le basi di tutti i triangoli: i due protoinsiemi sono mutuamente equivalenti.

Per ``chiudere il cerchio'', rimane da provare che il protoinsieme con triangoli aurei \eqref{eq:Rtriangles} e quello con semi-rombi sono mutuamente equivalenti.

Per passare da semi-aquilone e semi-dardo a semi-rombi, rimuoviamo la ``base'' di ciascun dardo:
\begin{equation}\label{eq:comp1}
\begin{tikzpicture}[baseline=(current bounding box.center),>=Straight Barb]

\coordinate (d) at (0,0);
\coordinate (e1) at (216:\mra);
\coordinate (e2) at (-36:\mra);
\coordinate (dp) at ($(e1)+(e2)$);

\begin{scope}[shift={(-2.4*\mra,0)}]
\coordinate (a) at (0,0);
\coordinate (b1) at (216:\mra);
\coordinate (b2) at (-36:\mra);
\coordinate (c) at ($(b2)+(-\mra,0)$);
\coordinate (ap) at ($(b1)+(b2)$);
\end{scope}

\path[semithick,gray,dashed,-To] ($(b2)+(0.1,0.2)$) edge[bend left=20] ($(e1)+(-0.1,0.2)$);

\fill[GSB] (a) -- (b1) -- (c) -- cycle;
\fill[GTA] (a) -- (c) -- (b2) -- cycle;
\fill[GSB] (d) -- (e1) -- (e2) -- cycle;
\fill[GSA] (ap) -- (b1) -- (c) -- cycle;
\fill[GTB] (ap) -- (c) -- (b2) -- cycle;
\fill[GSA] (dp) -- (e1) -- (e2) -- cycle;

\begin{scope}[decoration={
    markings,
    mark=at position 0.5 with {\arrow[scale=1.1,xshift=1.5pt]{>}}}
    ]
\draw[postaction={decorate}] (a) -- (b2);
\draw[postaction={decorate}] (d) -- (e2);
\draw[postaction={decorate}] (ap) -- (b2);
\draw[postaction={decorate}] (dp) -- (e2);
\draw[postaction={decorate}] (c) -- (b1);
\end{scope}

\begin{scope}[decoration={
    markings,
    mark=at position 0.5 with {\arrow[scale=1.1,xshift=3pt]{>>}}}
    ]
\draw[postaction={decorate}] (a) -- (b1);
\draw[postaction={decorate}] (c) -- (a);
\draw[postaction={decorate}] (d) -- (e1);
\draw[postaction={decorate}] (c) -- (ap);
\draw[postaction={decorate}] (dp) -- (e1);
\draw[postaction={decorate}] (ap) -- (b1);
\draw[postaction={decorate}] (e1) -- (e2);
\draw[postaction={decorate}] (b2) -- (c);
\end{scope}


\end{tikzpicture}
\end{equation}
Bisogna naturalmente provare che questa operazione è ben definita. Questo segue dal fatto che, in una tassellazione con semi-aquilone e semi-dardo, ogni semi-dardo è contenuto in uno ed un solo ``diamante'' (il motivo a sinistra nella figura). Per i dettagli della dimostrazione si rimanda a \cite{DanPen}. La sostituzione in figura elimina tutti i semi-dardi (ma non i semi-aquiloni) e produce una tassellazione con semi-rombi che rispetta le regole di corrispondenza (siccome le decorazioni sui lati esterni del diamante non vengono alterate).

In maniera simile il passaggio da semi-rombi a semi-dardo e semi-aquilone è basato sulla seguente sostituzione:
\begin{equation}\label{eq:comp2}
\begin{tikzpicture}[baseline=(current bounding box.center),>=Straight Barb]

\coordinate (v0) at (0,\mra);
\coordinate (v1) at (198:\mra);
\coordinate (v2) at (-18:\mra);
\coordinate (v3) at (0,0);
\coordinate (v4) at ($(v1)+(-18:\mra)$);

\begin{scope}[xshift=2.3*\mra]
\coordinate (w0) at (0,\mra);
\coordinate (w1) at (198:\mra);
\coordinate (w2) at (-18:\mra);
\coordinate (w4) at ($(w1)+(-18:\mra)$);
\end{scope}

\path[semithick,gray,dashed,-To] ($(v2)+(-0.5,1.2)$) edge[bend left=20] ($(w1)+(0.5,1.2)$);

\fill[GTA] (v2) -- (v3) -- (v4) -- cycle;
\fill[GTB] (v1) -- (v3) -- (v4) -- cycle;
\fill[GSA] (v0) -- (v2) -- (v3) -- cycle;
\fill[GSB] (v0) -- (v1) -- (v3) -- cycle;
\fill[GTA] (w0) -- (w2) -- (w4) -- cycle;
\fill[GTB] (w0) -- (w1) -- (w4) -- cycle;

\begin{scope}[decoration={
    markings,
    mark=at position 0.5 with {\arrow[scale=1.1,xshift=1.5pt]{>}}}
    ]
\draw[postaction={decorate}] (v3) -- (v1);
\draw[postaction={decorate}] (v3) -- (v2);
\draw[postaction={decorate}] (w4) -- (w0);
\end{scope}

\begin{scope}[decoration={
    markings,
    mark=at position 0.5 with {\arrow[scale=1.1,xshift=3pt]{>>}}}
    ]
\draw[postaction={decorate}] (v3) -- (v0);
\draw[postaction={decorate}] (v0) -- (v1);
\draw[postaction={decorate}] (v0) -- (v2);
\draw[postaction={decorate}] (v4) -- (v3);
\draw[postaction={decorate}] (v1) -- (v4);
\draw[postaction={decorate}] (v2) -- (v4);
\draw[postaction={decorate}] (w1) -- (w4);
\draw[postaction={decorate}] (w2) -- (w4);
\draw[postaction={decorate}] (w0) -- (w1);
\draw[postaction={decorate}] (w0) -- (w2);
\end{scope}


\end{tikzpicture}
\end{equation}
Si dimostra, in maniera analoga, che tale trasformazione è ben definita
e che il suo effetto è di rimuovere da ogni acutaureo il lato con singola freccia.
Il risultato è una tassellazione con i triangoli \eqref{eq:Rtriangles}, eccetto che entrambi sono riscalati di un fattore $\varphi$.

Applicando in successione le due sostituzioni qui sopra, otteniamo una operazione chiamata \emph{Composizione}. Tale operazione è invertibile e la sua inversa, che chiamiamo \emph{Decomposizione}, consiste nella seguente sostituzione:
\begin{center}
\begin{tikzpicture}[xscale=-1,>=Straight Barb]

\coordinate (a) at (0,0);
\coordinate (b1) at (216:\mra);
\coordinate (b2) at (-36:\mra);
\coordinate (c) at ($(b2)+(-\mra,0)$);

\begin{scope}[shift={(2.5*\mra,0)}]
\coordinate (d) at (0,0);
\coordinate (e1) at (216:\mra);
\coordinate (e2) at (-36:\mra);
\end{scope}

\path[semithick,gray,dashed,-To] ($(e1)+(0,0.5)$) edge[bend right=20] ($(b2)+(0,0.5)$);

\fill[GSA] (a) -- (b1) -- (c) -- cycle;
\fill[GTB] (a) -- (c) -- (b2) -- cycle;
\fill[GSA] (d) -- (e1) -- (e2) -- cycle;

\begin{scope}[decoration={
    markings,
    mark=at position 0.5 with {\arrow[scale=1.1,xshift=1.5pt]{>}}}
    ]
\draw[postaction={decorate}] (c) -- (b1);
\draw[postaction={decorate}] (a) -- (b2);
\draw[postaction={decorate}] (d) -- (e2);
\end{scope}

\begin{scope}[decoration={
    markings,
    mark=at position 0.5 with {\arrow[scale=1.1,xshift=3pt]{>>}}}
    ]
\draw[postaction={decorate}] (a) -- (b1);
\draw[postaction={decorate}] (b2) -- (c);
\draw[postaction={decorate}] (e1) -- (e2);
\draw[postaction={decorate}] (c) -- (a);
\draw[postaction={decorate}] (d) -- (e1);
\end{scope}


\end{tikzpicture}

\begin{tikzpicture}[>=Straight Barb]

\clip (-5.2,-0.2) rectangle (2.2,3.2);

\coordinate (b1) at (0,0);
\coordinate (b2) at (\mra,0);
\coordinate (f) at ($(b2)+(108:\mra)$);
\coordinate (c) at ($(b2)+(144:\mra)$);
\coordinate (a) at ($(c)+(72:\mra)$);

\begin{scope}[shift={(-2.5*\mra,0)}]
\coordinate (e1) at (0,0);
\coordinate (e2) at (\mra,0);
\path[name path=PA] (e1) -- ++(72:4.5);
\path[name path=PB] (e2) -- ++(108:4.5);
\path[name intersections={of=PA and PB,by=d}];
\end{scope}

\path[semithick,gray,dashed,-To] ($(e2)+(0.3,1.7)$) edge[bend left=20] ($(b1)+(-0.3,1.7)$);

\fill[GTA] (b1) -- (b2) -- (c) -- cycle;
\fill[GTB] (b2) -- (c) -- (f) -- cycle;
\fill[GSA] (c) -- (f) -- (a) -- cycle;
\fill[GTA] (d) -- (e1) -- (e2) -- cycle;

\begin{scope}[decoration={
    markings,
    mark=at position 0.5 with {\arrow[scale=1.1,xshift=1.5pt]{>}}}
    ]
\draw[postaction={decorate}] (c) -- (b2);
\draw[postaction={decorate}] (f) -- (a);
\draw[postaction={decorate}] (e1) -- (d);
\end{scope}

\begin{scope}[decoration={
    markings,
    mark=at position 0.5 with {\arrow[scale=1.1,xshift=3pt]{>>}}}
    ]
\draw[postaction={decorate}] (b2) -- (b1);
\draw[postaction={decorate}] (b2) -- (f);
\draw[postaction={decorate}] (c) -- (a);
\draw[postaction={decorate}] (d) -- (e2);
\draw[postaction={decorate}] (b1) -- (c);
\draw[postaction={decorate}] (f) -- (c);
\draw[postaction={decorate}] (e2) -- (e1);
\end{scope}


\end{tikzpicture}
\end{center}
In Figura~\ref{fig:penrose} vediamo una tassellazione con semi-aquiloni e semi-dardi in rosso, ed in bianco e grigio i tasselli che si ottengono applicando la Decomposizione.

La Decomposizione è ben definita anche quando applicata a tassellazioni il cui supporto non è l'intero piano ma un  sottoinsieme, ed usando tale operazione possiamo mostrare che \emph{esistono tassellazioni di Penrose del piano} costruendone una esplicitamente. Iniziamo dalla seguente tassellazione, il cui supporto è un aquilone:
\[
\mathcal{C}_0:=
\begin{tikzpicture}[baseline=(current bounding box.center),font=\small,>=Straight Barb]

\clip (-1.5,-0.2) rectangle (1.5,2.3);

\coordinate (a4) at (0,0);
\coordinate (a2) at (162:\raggio);
\coordinate (a3) at (18:\raggio);
\path[name path=PB] (a2)-- ++(54:3);
\path[name path=PA] (a3)-- ++(126:3);
\path[name intersections={of=PA and PB,by=a1}];

\fill[GTA] (a1) -- (a2) -- (a4) -- cycle;
\fill[GTB] (a1) -- (a3) -- (a4) -- cycle;

\begin{scope}[decoration={
    markings,
    mark=at position 0.5 with {\arrow[scale=1.1,xshift=1.5pt]{>}}}
    ]
\draw[postaction={decorate}] (a2) -- (a1);
\draw[postaction={decorate}] (a3) -- (a1);
\end{scope}

\begin{scope}[decoration={
    markings,
    mark=at position 0.5 with {\arrow[scale=1.1,xshift=3pt]{>>}}}
    ]
\draw[postaction={decorate}] (a1) -- (a4);
\draw[postaction={decorate}] (a4) -- (a2);
\draw[postaction={decorate}] (a4) -- (a3);
\end{scope}

\draw[red,dashed] (a2) -- ++(18:\raggio) -- (a3);
\shade[ball color=red] ($(a2)+(18:\raggio)$) circle (0.06);
\end{tikzpicture}
\]
Notiamo che non si tratta del prototassello in \eqref{eq:KD}: i triangoli sono attaccati lungo il lato con doppia freccia.
Il pallino rosso indica la posizione dell'origine di $\R^2$ (il punto fisso rispetto alle dilatazioni).

D'ora in poi omettiamo le frecce: disegneremo in bianco i triangoli le cui doppie frecce vanno in senso orario, ed in grigio quelli le cui doppie frecce vanno in senso antiorario. Riscalando $\mathcal{C}_0$ di un fattore $\varphi$ ed applicando la Decomposizione otteniamo:
\[
\mathcal{C}_1:=
\includegraphics[page=2,valign=c]{cartwheel.pdf}
\]
Riscalando di nuovo di $\varphi$ e applicando la Decomposizione otteniamo:
\[
\mathcal{C}_2:=
\includegraphics[page=3,valign=c]{cartwheel.pdf}
\]
Costruiamo in questo modo una successione $(\mathcal{C}_0,\mathcal{C}_1,\mathcal{C}_2,\ldots)$ di tassellazioni, tutte con lo stesso protoinsieme \eqref{eq:Rtriangles}, e con supporto aquiloni sempre più grandi.
Osserviamo che $\mathcal{C}_0\subset\mathcal{C}_2$ e, siccome per costruzione la Decomposizione commuta con le similitudini del piano, $\mathcal{C}_n\subset\mathcal{C}_{n+2}$ per ogni $n\in\N$. Ne consegue che l'unione
\[
\mathcal{C}:=\bigcup_{n\in 2\N+1}\mathcal{C}_n
\]
è una tassellazione del piano, detta \emph{Cartwheel}. Il metodo con cui è ottenuta tale tassellazione (dilatazione e Decomposizione ripetute) è detto di \emph{sostituzione}.
Non tutte le tassellazioni di Penrose si possono ottenere per sostituzione!

Gran parte delle proprietà locali delle tassellazioni di Penrose si dimostrano utilizzando la Composizione, a cominciare dall'aperiodicità. Supponiamo per assurdo che $\mathcal{T}$ sia una tassellazione di Penrose e che
\[
\mathcal{T}+\vec{v}=\mathcal{T}
\]
per un qualche $\vec{v}\in\R^2$ diverso da zero. Poniamo $\mathcal{T}_0:=\mathcal{T}$ e definiamo una successione $(\mathcal{T}_n)_{n\in\N}$ di tassellazioni di Penrose per induzione in modo che $\mathcal{T}_{n+1}$ sia ottenuto da $\mathcal{T}_n$ con una Composizione ed una dilatazione di $\varphi^{-1}$ (per riportare i tasselli alla dimensione originale). La Composizione commuta con le traslazioni, ma le dilatazioni no. Per ogni $n\in\N$ si ha:
\[
\mathcal{T}_n+\varphi^{-n}\vec{v}=\mathcal{T}_n .
\]
Siccome $\varphi^{-n}\to 0$ per $n\to \infty$, otteniamo tassellazioni con simmetrie traslazionali arbitrariamente piccole. Questo è impossibile in quanto qualunque traslazione di una lunghezza inferiore al diametro del cerchio inscritto nell'ottusaureo non può essere una simmetria.
Abbiamo in questo modo provato che:

\begin{thm}
Il protoinsieme \eqref{eq:Rtriangles} è aperiodico.
\end{thm}

Le tassellazioni del piano con protoinsieme \eqref{eq:Rtriangles} (nel seguito dette semplicemente \emph{tassellazioni di Penrose}) hanno numerose proprietà interessanti:
\begin{enumerate}
\item\label{enum1} Esistono una infinità non numerabile di tassellazioni di Penrose inequivalenti (due tassellazioni si dicono \emph{equivalenti} se possono essere trasformate una nell'altra con una isometria del piano). 

\item\label{enum2} A meno di una equivalenza, esistono esattamente due tassellazioni di Penrose con simmetria 5-fold (una è mostrata in Figura~\ref{fig:5fold}.

\item\label{enum3} Le tassellazioni di Penrose sono \emph{ripetitive} (ogni porzione finita è ripetuta traslata infinite volte nella tassellazione).

\item\label{enum4} Tutte le tassellazioni di Penrose sono \emph{localmente equivalenti} (ogni porzione finita di una tassellazione compare, ripetuta infinite volte, in ogni altra tassellazione).

\item 
Il Cartwheel, che in Figura~\ref{fig:worms} vediamo realizzato con dardi e aquiloni, è invariante per riflessioni rispetto all'asse verticale.
\end{enumerate}

Da \ref{enum2} e \ref{enum4} segue che ogni tassellazione di Penrose \emph{localmente} ha simmetria 5-fold: ovvero è possibile trovare al suo interno sottoinsiemi finiti arbitrariamente grandi con simmetria 5-fold. 

Da \ref{enum4} segue che non è possibile, guardando solo una porzione finita di una tassellazione, capire di quale tassellazione di Penrose si tratta (ad esempio se è il Cartwheel o un'altra tassellazione). Come si fa allora a dimostrare che ne esistono infinite inequivalenti? Il trucco è convertire le tassellazioni in successioni binarie.

Data una tassellazione $\mathcal{T}=\mathcal{T}_0$ di Penrose, scegliamo un punto $x_0$ interno ad un tassello. Applichiamo ripetutamente i due passi \eqref{eq:comp1} e \eqref{eq:comp2} della Composizione e costruiamo in questo modo una successione $(\mathcal{T}_n)_{n\in\N}$ di tassellazioni. Ciascuna tassellazione ha due prototasselli, uno di area maggiore dell'altro. Chiamiamo \emph{largo} il prototassello di area maggiore, e \emph{stretto} l'altro.
Per ogni $n$, esiste un solo tassello di $\mathcal{T}_n$ che ha $x_0$ al suo interno. Poniamo $a_n:=0$ se tale tassello è largo, ed $a_n:=1$ se è stretto. In questo modo trasformiamo la coppia $(\mathcal{T},x_0)$ in una successione binaria $(a_0,a_1,a_2,\ldots)$ detta \emph{successione di indici} (\emph{index sequence}). Per come è definita la Composizione, non è difficile convincersi che in tale successione binaria se $a_n=1$ allora $a_{n+1}=0$ (la stringa $11$ non appare mai nella successione).

Per ottenere una corrispondenza fra tassellazioni di Penrose e successioni binarie, dobbiamo ora eliminare la dipendenza della successione di indici dal punto di base. Per fare questo, introduciamo una relazione di equivalenza.

Diciamo che due successioni $(a_0,a_1,\ldots)$ e $(b_0,b_1,\ldots)$ hanno la \emph{stessa coda} se esiste $N$ intero positivo tale che $a_n=b_n$ per ogni $n\geq N$. La classe di equivalenza di una successione rispetto a tale relazione è detta la sua \emph{coda}.

La successione di indici di una tassellazione dipende dalla scelta del punto di base $x_0$, ma si può provare che la sua coda è indipendente da tale scelta, ed infatti che due tassellazioni sono equivalenti se e solo se le loro successioni di indici hanno la stessa coda. Si può inoltre dimostrare che ogni successione binaria che non contenga la stringa $11$ è la successione di indici di una tassellazione di Penrose. Arriviamo allora al seguente risultato fondamentale:

\begin{thm}
Esiste una corrispondenza biunivoca fra classi di equivalenza di tassellazioni di Penrose e code di successioni binari che non contengono la stringa $11$.
\end{thm}

E' ora un semplice esercizio provare che l'insieme delle code di successioni qui sopra ha la potenza del continuo. Le tassellazioni che si ottengono per sostituzione sono quelle che hanno successione di indici periodica. Il Cartwheel, ad esempio, ha successione di indici identicamente nulla.

Per uno studio più approfondito delle tassellazioni di Penrose si rimanda a \cite{DanPen}.

\section{Pentagriglie}\label{sec:pen2}

Abbiamo già osservato come la colorazione dei rombi 
in Figura~\ref{fig:romper} suggerisca che i tasselli siano ottenuti
proiettando dei cubi unitari, e abbiamo mostrato come sia possibile ottenere
una tassellazione aperiodica proiettando parte di una periodica se la direzione
della proiezione è scelta bene (Figura~\ref{fig:staircase}).

Consideriamo di nuovo i rombi di Penrose \eqref{eq:Krombi},
\begin{center}
\begin{tikzpicture}[baseline=(current bounding box.center),font=\small,>=Straight Barb]

\coordinate (a3) at (-4,0);
\coordinate (a1) at ($(a3)+(-18:1)$);
\coordinate (a4) at ($(a3)+(18:1)$);
\coordinate (a2) at ($(a1)+(a4)-(a3)$);

\coordinate (b3) at (0,0);
\coordinate (b1) at (-54:1);
\coordinate (b4) at (54:1);
\coordinate (b2) at ($(b1)+(b4)$);

\fill[brown!20] (a1) -- (a3) -- (a4) -- (a2) -- cycle;
\fill[brown!20] (b1) -- (b3) -- (b4) -- (b2) -- cycle;

\begin{scope}[decoration={
    markings,
    mark=at position 0.5 with {\arrow[scale=1.1,xshift=1pt]{>}}}
    ]
\draw[postaction={decorate}] (a4) -- (a2);
\draw[postaction={decorate}] (a4) -- (a3);
\draw[postaction={decorate}] (b2) -- (b4);
\draw[postaction={decorate}] (b3) -- (b4);
\end{scope}

\begin{scope}[decoration={
    markings,
    mark=at position 0.5 with {\arrow[scale=1.1,xshift=3pt]{>>}}}
    ]
\draw[postaction={decorate}] (a2) -- (a1);
\draw[postaction={decorate}] (a3) -- (a1);
\draw[postaction={decorate}] (b2) -- (b1);
\draw[postaction={decorate}] (b3) -- (b1);
\end{scope}

\end{tikzpicture}
\end{center}
Questi due prototasselli rappresentano una modifica minimale dei rombi in Figura~\ref{fig:romper}: tutti i lati hanno la stessa lunghezza, e cambia solo la scelta degli angoli, multipli opportuni di $\pi/10$ (compaiono tutti i multipli da $1$ a $4$).

Una porzione di tassellazione del piano con questi rombi, costruita rispettando le regole di corrispondenza, ha l'aspetto in Figura~\ref{fig:rombiproj}.

\begin{figure}[H]
\includegraphics[width=0.8\columnwidth]{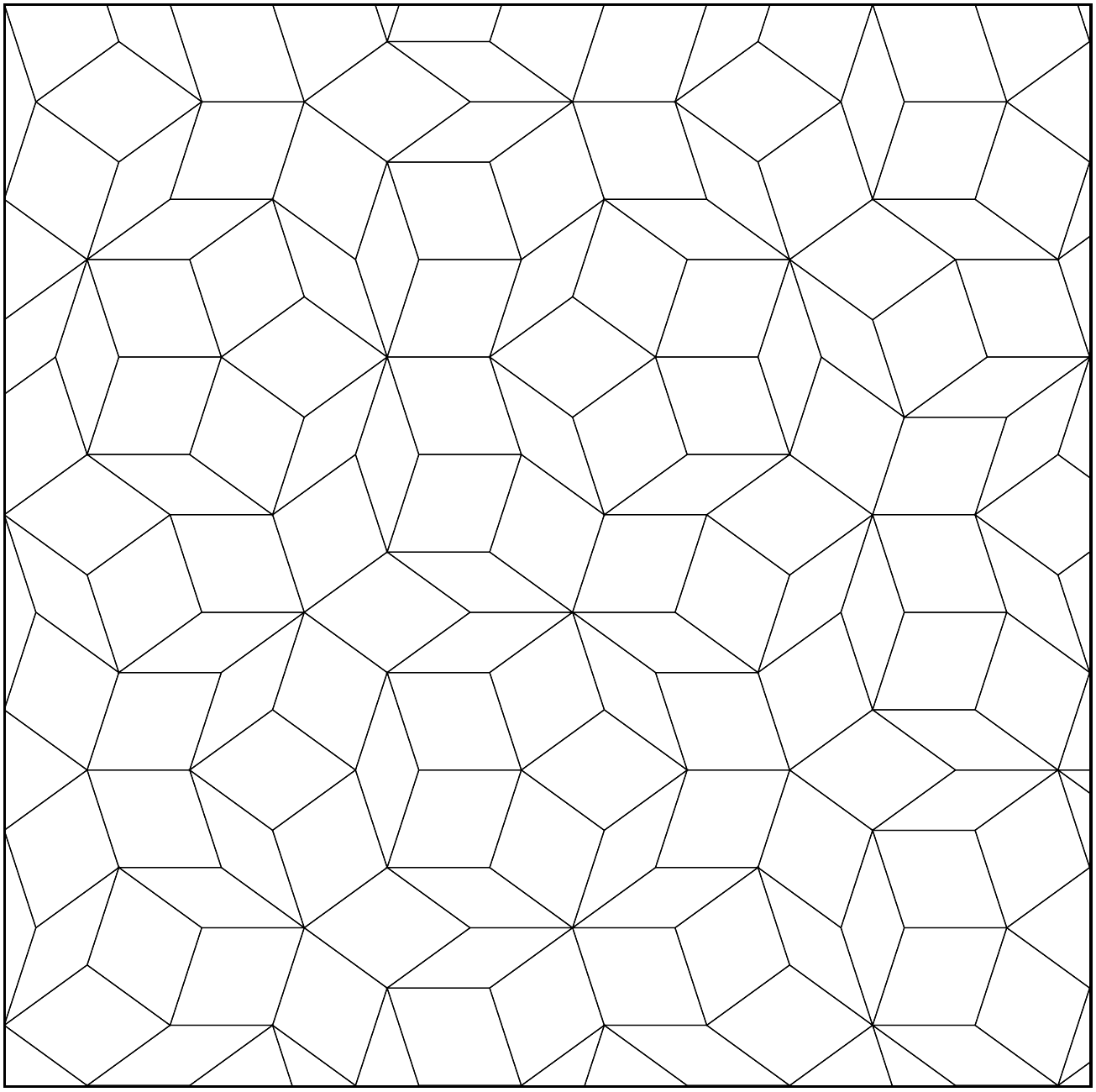}

\caption{Tassellazione con rombi di Penrose}\label{fig:rombiproj}
\end{figure}

\pagebreak

\end{multicols}

\begin{figure}[H]
\includegraphics[height=0.43\textheight]{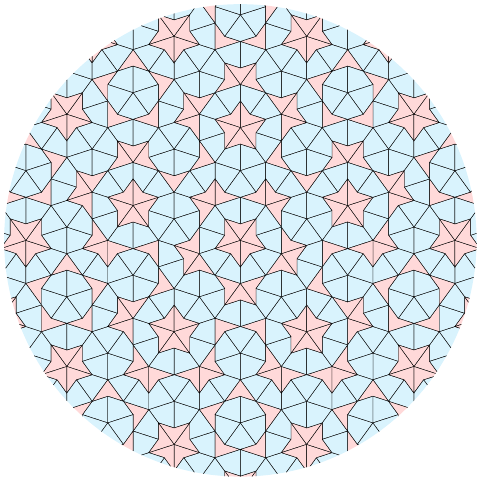}

\caption{Una tassellazione di Penrose con simmetria 5-fold}\label{fig:5fold}
\end{figure}

\begin{figure}[H]
\includegraphics[height=0.44\textheight,rotate=180]{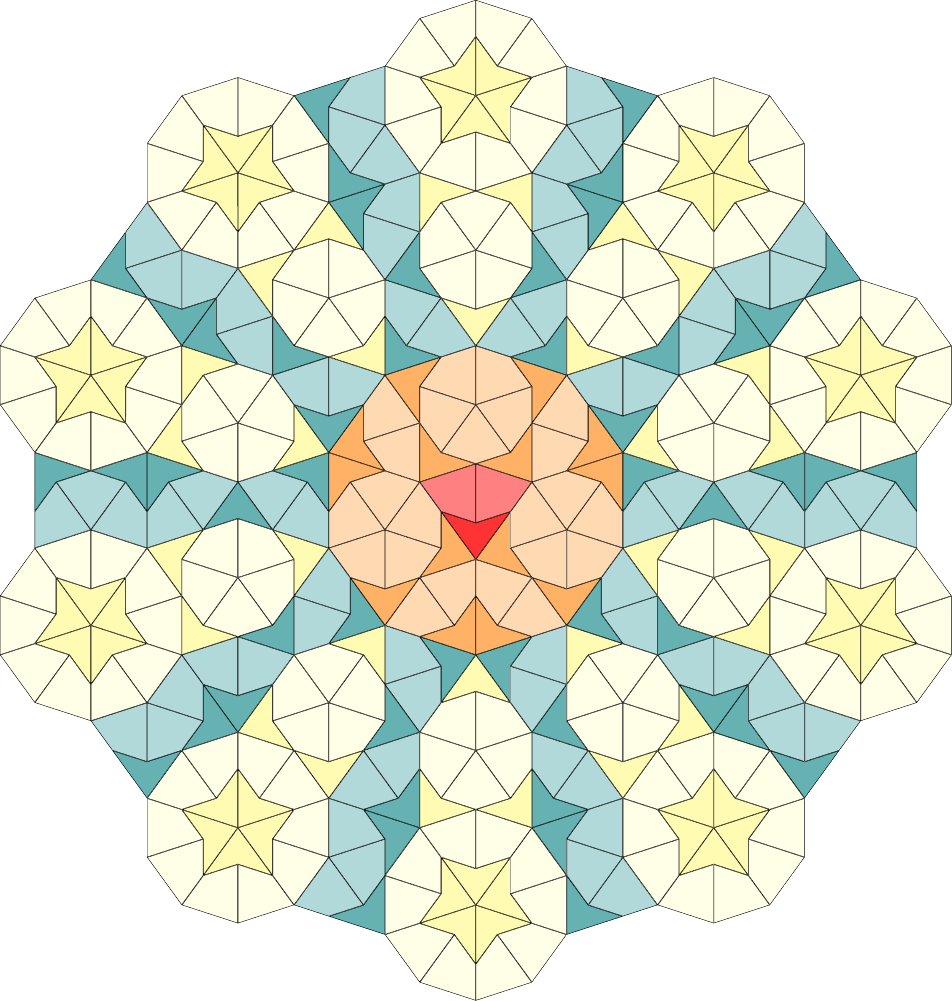}

\caption{Il Cartwheel con dardi e aquiloni}\label{fig:worms}
\end{figure}

\begin{multicols}{2}

Anche in Figura \ref{fig:rombiproj} ci sembra di scorgere proiezioni di cubi unitari, sebbene combinati in modo inusuale, impossibile se fossero proiettati dallo spazio tridimensionale. Ricorda vagamente la litografia \textit{Relatività} (o \textit{Casa di scale}, 1953) di Escher.

Possiamo ottenere tassellazioni di Penrose proiettando cubi unitari?
La risposta a questa domanda è contenuta nel lavoro di N.G.~de Bruijn \cite{dB81}, il quale ha dato un contributo fondamentale allo studio delle tassellazioni di Penrose. Alla base della sua costruzione c'è la nozione di pentagriglia.
Poniamo
\[
\zeta:=e^{2\pi\mathrm{i}/5} ,
\]
radice quinta primitiva dell'unità, e identifichiamo il piano complesso $\C$ con $\R^2$ nel modo usuale.
Dato $\vec{\gamma} = (\gamma_0,\gamma_1,\ldots,\gamma_4) \in\R^5$ tale che
\[
\gamma_0+\ldots+\gamma_4=0 ,
\]
chiamiamo \emph{pentagriglia} di parametro $\vec{\gamma}$ l'insieme
\[
\bigcup_{j=0}^4 \big\{z\in\C:\mathrm{Re}(z\zeta^{-j})+\gamma_j\in\Z\big\} .
\]
Si tratta di cinque fasci di rette parallele, cinque \emph{griglie}, con direzioni ortogonali a quelle dei lati in Figura~\ref{fig:rombiproj}. In ciascuna griglia la distanza fra due rette consecutive è unitaria.

Una pentagriglia si dice \emph{regolare} se in ciascun punto del piano si intersecano al più due sue rette. Una pentagriglia si dice \emph{singolare} se non è regolare.

Con le potenze di $\zeta$ possiamo costruire una base ortonormale dello spazio vettoriale $\C^5$, formata dai cinque vettori
\[
\vec{e}_j:=\frac{1}{\sqrt{5}}(1,\zeta^j,\zeta^{2j},\zeta^{3j},\zeta^{4j}) , \qquad j=-2,\ldots,2.
\]
In $\R^5\subset\C^5$ consideriamo il piano $\Pi_\gamma$ formato dai punti di coordinate
\[
z\vec{e}_1+\bar{z}\vec{e}_{-1}+\vec{\gamma} , \quad z\in\C
\]
(in cui $\bar z$ è il complesso coniugato di $z$) e consideriamo la proiezione ortogonale $\R^5\to \Pi_\gamma$ su tale piano.

In modo simile a quanto fatto nell'esempio in Figura~\ref{fig:staircase}, suddividiamo $\R^5$ in ipercubi unitari. Per ogni 
$\vec{n}=(n_0,\ldots,n_4)\in\Z^5$, consideriamo l'ipercubo unitario (aperto)
\[
C(\vec{n}):=\big\{x\in\R^5: n_i-1<x_i<n_i\;\forall\;i=0,\ldots,4\big\} .
\]
Se l'ipercubo $C(\vec{n})$ interseca il piano $\Pi_\gamma$, allora proiettiamo tale punto $\vec{n}$ sul piano stesso.

\begin{thm}
I punti ottenuti in questo modo sono i vertici di una tassellazione del piano $\Pi_\gamma$ con rombi di Penrose (di lato di lunghezza $\sqrt{2}$), e la tassellazione può essere ricostruita univocamente a partire dal suo insieme di vertici.
\end{thm}

La dimostrazione di questo teorema è piuttosto elaborata, e può essere trovata nell'articolo originale di de Brujin
\cite{dB81}, oppure in \cite{DanPen} (in particolare, si veda la Sez.~5.2 e la Prop.~4.6). Si può inoltre mostrare che:
\begin{enumerate}
\item Una tassellazione ottenuta in questo modo dipende da $\vec{\gamma}$ solo attraverso il numero complesso
\[
\omega_\gamma:=\sum_{j=0}^4\gamma_j\zeta^{2j} .
\]
\item
Due tassellazioni associate a pentagriglie $\vec{\gamma}$ e $\vec{\gamma}'$ sono uguali se e solo se $\omega_\gamma=\omega_{\gamma'}$.

\item
Due tassellazioni associate a pentagriglie $\vec{\gamma}$ e $\vec{\gamma}'$ sono equivalenti se e solo se, per un qualche $k\in\{0,\ldots,4\}$, la differenza $\omega_\gamma-\zeta^{2k}\omega_{\gamma'}$ appartiene all'ideale $\mathcal{I}$ in $\Z[\zeta]$ generato da $1-\zeta$.

\item
Una pentagriglia $\vec{\gamma}$ è singolare se e solo se, per un qualche $\lambda\in\R$ e $k\in\{0,\ldots,4\}$,
si ha $\omega_\gamma-\mathrm{i}\lambda\zeta^k\in\mathcal{I}$ .
\end{enumerate}
Il lavoro \cite{dB81} è il punto di partenza per l'applicazione della Teoria Algebrica dei Numeri alle tassellazioni del piano.

Un risultato importante della teoria sviluppata da de Bruijn è che ogni tassellazione di Penrose si può ottenere a partire da una pentagriglia, se si tiene conto anche delle pentagriglie singolari (ma la costruzione in questo caso diventa più complicata).

Cosa c'è di magico nel numero $5$? Per $n>1$, l'anello
\[
\Z[e^{2\pi\mathrm{i}/n}]
\]
delle combinazioni intere delle potenze di $e^{2\pi\mathrm{i}/n}$ è detto \emph{anello ciclotomico}. Identificando $\C$ con $\R^2$, il sottoinsieme dei punti del piano appartenenti a tale anello è un reticolo se $n\in\{2,3,4,6\}$, ed ha immagine densa in $\R^2$ in tutti gli altri casi. 
Notiamo che i valori di $n$ per cui si ha un reticolo sono gli stessi del teorema di restrizione cristallografica. Si può immaginare di ripetere la costruzione di de Brujin proiettando da $\R^n$ con $n\geq 7$, per ottenere tassellazioni con rombi aperiodiche e localmente con simmetria rotazionale $n$-fold (ad esempio $8$-fold). Simmetrie del genere hanno interesse in chimica perché si osservano nei \emph{quasi-cristalli}, ma quella dei quasi-cristalli è tutta un'altra storia (raccontata, ad esempio, in \cite{Fer18}). La scoperta dei quasi-cristalli negli anni `80 ad opera di Dan Shechtman portò un rinnovato interesse per le tassellazioni. Lo spettro di diffrazione degli elettroni in un quasi-cristallo ha una struttura priva di periodicità (al contrario dei solidi cristallini) che si può ottenere immaginando che gli atomi siano disposti nei vertici di una tassellazione poligonale aperiodica. La definizione e lo studio di ``quasi-cristalli matematici'' ed il loro legame con le tassellazioni si può trovare ad esempio in \cite{Baa02}.

\section{Tassellazioni e dinamica simbolica}\label{sec:simdin}

Un \emph{sistema dinamico discreto} è una coppia $(X,\sigma)$, in cui $X$ è uno spazio topologico e $\sigma:X\to X$ un omeomorfismo locale. Interpretiamo i punti di $X$ come ``stati'' del sistema fisico che ci interessa studiare, e $\sigma$ come generatore di una evoluzione temporale: un sistema che al tempo $0$ è nello stato $x_0\in X$,
al tempo $n$ si troverà nello stato $\sigma^n(x_0)$. (Attenzione: qui il termine ``discreto'' si riferisce alla variabile tempo, non alla topologia di $X$.)
In termini più astratti, $\sigma$ genera una azione del semigruppo $(\N,+,0)$ su $X$. Nel caso in cui $\sigma$ è un omeomorfismo, e quindi invertibile, questa si estende ad una azione del gruppo $(\Z,+,0)$. Nel secondo caso, chiamiamo \emph{orbita} di $x\in X$ l'insieme $\{\sigma^n(x)\mid n\in\Z\}$.
Nel primo caso, esistono vari nozioni di orbita di un punto $x$ per l'azione di un semigruppo e quella che interessa a noi è l'insieme $\{\sigma^n(x)\mid n\in\N\}$, detto a volte \emph{orbita unilaterale}.

In analogia con le tassellazioni, una tipica domanda che ci si può porre
su un sistema dinamico è:
\begin{equation}\label{eq:perorb}
\text{Esistono orbite periodiche?}
\end{equation}

Una famiglia di sistemi dinamici particolarmente interessante è la seguente.
Dato un insieme finito e non vuoto $A$ (che chiameremo \emph{alfabeto}), possiamo considerare l'insieme $A^{\N}$ di tutte le successioni di elementi di $A$.
Date due successioni $a=(a_0,a_1,a_2,\ldots)$ e $b=(b_0,b_1,b_2,\ldots)$, se $n$ è il più piccolo intero tale che $a_{n-1}=b_{n-1}$ ed $a_n\neq b_n$, poniamo $d(a,b):=2^{-n}$. Poniamo $d(a,b):=0$ se $a=b$, e $d(a,b):=1$ se $a_0\neq b_0$. La funzione $d$ così definita è una metrica su $A^{\N}$, ed induce una topologia che si può mostrare essere la topologia prodotto costruita mettendo su $A$ la topologia discreta.

La funzione $d$ è in effetti una \emph{ultrametrica}: per ogni $a,b,c\in X$,
si ha
\[
d(a,c)\leq\max\big\{d(a,b),d(b,c)\big\} .
\]
Gli spazi ultrametrici hanno molte proprietà particolari, che si riflettono sulla loro topologia. Ad esempio:
\begin{itemize}
\item Due palle aperte con lo stesso raggio si intersecano se e solo se sono uguali.
\item Data una palla aperta, ogni suo punto è un suo centro.
\item Se $X$ è compatto, ogni palla aperta è anche chiusa, e l'insieme delle palle di raggio $r>0$ fissato è una partizione di $X$.
\end{itemize}
In particolare, dall'ultima proprietà segue che uno spazio ultrametrico compatto è totalmente disconnesso.
Se, in aggiunta, $X$ non ha punti isolati, allora è omeomorfo all'\emph{insieme di Cantor}.
(Quest'ultima affermazione è dimostrata ad esempio in \cite[Cap.~1]{Put18}, mentre le altre dimostrazioni sono un semplice esercizio.)

Nell'insieme $A^{\N}$, un omeomorfismo locale è dato dalla funzione che elimina il primo termine di una successione:
\[
\sigma(a_0,a_1,a_2,\ldots):=(a_1,a_2,\ldots) .
\]
In generale, un sottoinsieme chiuso $X\subseteq A^{\N}$ è detto \emph{subshift} se $\sigma(X)\subseteq X$. In questo caso, possiamo restringere $\sigma$ ad $X$ ed ottenere un sistema dinamico $(X,\sigma)$. Tali sistemi dinamici sono l'oggetto di studio di quel campo della matematica chiamato Dinamica Simbolica~\cite{LM95}.

\smallskip

Non è difficile adattare le costruzioni e le definizioni appena date al caso delle successioni bilatere $A^{\Z}$. La distanza fra $a,b\in A^{\Z}$ è definita come l'inf di $2^{-|n|}$ fatto su tutti gli $n\in\Z$ tali che $a_n\neq b_n$, l'applicazione $\sigma$ è data da $\sigma(a)_n:=a_{n+1}$ (per ogni $a\in X$ ed $n\in\Z$) ed è invertibile (quindi un omeomorfismo, piuttosto che solamente un omeomorfismo locale). La nozione di ``subshift'' di $A^{\Z}$ è analoga a quella per $A^{\N}$.

\begin{ex}[Domino]
Sia $A$ l'insieme delle tre tessere del domino
$(3, 2)$, $(1, 3)$ e $(2, 1)$.
L'insieme delle tassellazioni unidimensionali con queste tre tessere è
un subshift di $A^{\Z}$. Ad ogni tassellazione possiamo associare la successione che in posizione $n\in\Z$ ha la tessera del domino centrata nel punto di coordinata $2n$ (ciascuna tessera ha lunghezza $2$).
Tale subshift è formato da tutte le successioni che non contengono le combinazioni proibite
\[
(3,2)(1,3) \qquad\quad
(1,3)(2,1) \qquad\quad
(2,1)(3,2) .
\]
\end{ex}

\smallskip

\begin{ex}[Penrose]
L'insieme $X\subset \{0,1\}^{\N}$ delle successioni in cui ogni $1$ è seguito da $0$ è un 
subshift. Si tratta dell'insieme delle successioni di indici che parametrizzano le tassellazioni di Penrose. Notiamo che una successione appartiene ad $X$ se e solo se non contiene la stringa $11$: in questo senso, la stringa $11$ è una \emph{parola proibita}.
Notiamo inoltre che due successioni $a,b$ hanno la stessa coda se e solo se $\sigma^n(a)=\sigma^n(b)$ per un qualche $n\geq 0$ (dopo un certo lasso di tempo evolvono nella stessa successione).
\end{ex}

Per semplicità di trattazione, dimentichiamo le successioni bilatere e concentriamoci sui subshift di $A^{\N}$.

Sia $F$ un insieme numerabile di parole nell'alfabeto $A$, che chiameremo \emph{parole proibite}. Non è difficile mostrare che l'insieme
\[
X_F:=\big\{a\in A^{\N}\;\big|\;a \text{ non contiene stringhe di }F\big\}
\]
è un subshift. Viceversa, se $X$ è un subshift, esiste $F$ numerabile tale che $X=X_F$
(si prenda ad esempio come insieme $F$ la collezione di tutte le parole nell'alfabeto $A$
che non compaiono in alcuna successione di $X$).
Se $X=X_F$ diciamo che il subshift $X$ è \emph{definito} dall'insieme $F$.

\begin{df}
Un subshift è detto di \emph{tipo finito} (STF) se è definito da un insieme \underline{\smash{finito}} di parole proibite.
\end{df}

Non ogni subshift è di tipo finito. I STF si possono generare usando un \emph{automa a stati finiti} in maniera simile a come generavamo tassellazioni con pezzi del domino usando dei grafi. La definizione di automa va oltre gli scopi di questo saggio, e si può trovare ad esempio in \cite{LM95}.

Se $(X,\sigma)$ è un subshift, una domanda naturale è:
\begin{equation}\label{eq:persuc}
\text{Esistono in $X$ successioni periodiche?}
\end{equation}
Notiamo che le domande \eqref{eq:perorb} e \eqref{eq:persuc} sono equivalenti. L'orbita passante per un punto $x\in X$ è periodica se e solo se $\sigma^n(x)=x$ per un qualche $n\geq 0$.
Ma questo equivale a dire che $x$ è dato dalla stringa $(x_0,\ldots,x_{n-1})$ ripetuta all'infinito, ovvero la successione $x$ è periodica.

Un'altra domanda naturale, dato un insieme di parole proibite $F$, è:
\begin{equation}\label{eq:vuoto}
\text{Il subshift definito da $F$ è non vuoto?}
\end{equation}

In maniera molto simile a quanto fatto nel caso del domino, si può mostrare che:

\begin{prop}\label{prop:STFper}
Ogni STF non-vuoto contiene almeno una successione periodica.
\end{prop}

\begin{proof}
L'argomento è elementare.
Sia $X$ un subshift non vuoto definito da un insieme finito di parole proibite, la più lunga delle quali ha lunghezza $k$, e sia $x\in X$. Siccome esistono un numero finito di parole di lunghezza $k$, esiste una stringa di lunghezza $k$ ripetuta in $x$ almeno due volte. Cioè, esiste una stringa in $x$ della forma $aba$, con $a$ di lunghezza $k$ e $b$ di lunghezza opportuna. Tale stringa non contiene parole proibite, il che prova che la successione periodica
\[
abababab\ldots
\]
non contiene parole proibite, ovvero appartiene ad $X$.
\end{proof}

Un teorema analogo a quello di estensione per tassellazioni del piano garantisce
che, dato un insieme finito di parole proibite $F$, se è possibile scrivere stringhe arbitrariamente lunghe che non contengono parole proibite, allora $X_F$ è non vuoto \cite{LM95}. Da questo, e dalla proposizione Prop.~\ref{prop:STFper}, segue che esiste un algoritmo che permette di stabilire in un tempo finito se, dato $F$, il subshift $X_F$ è vuoto oppure no.

Come nel caso del domino e delle tassellazioni di Wang, la situazione cambia drasticamente se consideriamo la versione bidimensionale del problema, ovvero consideriamo l'insieme $A^{\N^2}$ di tutte le funzioni $\N^2\to A$. Possiamo visualizzare gli elementi di $A^{\N^2}$ come matrici infinite di elementi di $A$, e c'è una ovvia generalizzazione dello shift $\sigma$ ad una coppia di shift di $A^{\N^2}$ in direzione orizzontale e verticale.
Un subshift, in questo caso, sarà un sottoinsieme $X\subseteq A^{\N^2}$ chiuso ed invariante rispetto ad entrambi questi shift.
Si può provare che:

\begin{thm}
Esiste un SFT bidimensionale non vuoto che non possiede alcuna configurazione periodica.
\end{thm}

Una conseguenza di questo teorema è che non esiste nessun algoritmo che permetta di stabilire (in un tempo finito) se un STF bidimensionale è vuoto oppure no.

Per i dettagli rimandiamo al testo \cite{LM95}.

\section{Ein stein, una pietra}\label{sec:einstein}

Le tassellazioni di Penrose hanno numerose proprietà interessanti.
Il loro fascino, inoltre, è innegabile. Un amante delle simmetrie potrebbe passare ore ad osservarle cercando di scorgere strutture e regolarità. Cercando di scorgere ordine in un caos che è solo apparente.
Il valore estetico delle tassellazioni di Penrose è testimoniato anche dai pavimenti ad esse ispirati, che si possono trovare in varie università in giro per il mondo (a Oxford, Miami, Western Australia, Texas A{\&}M University, e in molte altre ancora).

La domanda lasciata naturalmente aperta dalla scoperta del protoinsieme di Penrose è se fosse possibile trovare un protoinsieme aperiodico formato da un solo prototassello, o \emph{ein stein} (``una pietra'', in tedesco).
Una prima risposta affermativa è stata data da Socolar e Taylor in \cite{ST12}. Il prototassello da loro inventato è un semplice esagono regolare, ma decorato in modo che opportune regole di corrispondenza forzino l'aperiodicità delle tassellazioni. Tassellazioni equivalenti, in cui l'aperiodicità è forzata dalla sola forma dei tasselli, si ottengono usando il prototassello nella figura seguente:
\begin{center}
\tikz{\node[xscale=-1,inner sep=-7pt]{\includegraphics[scale=0.22]{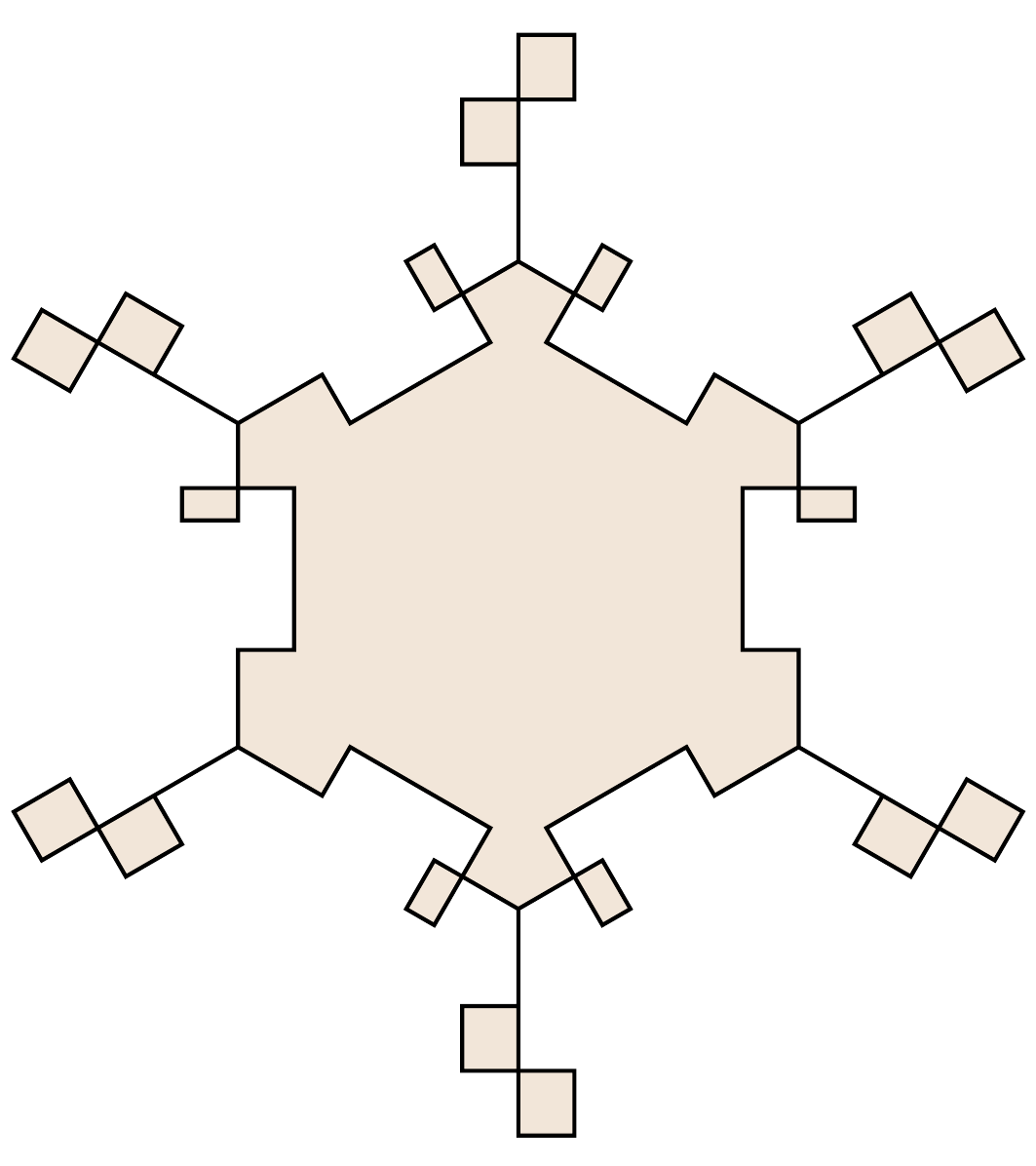}}}
\end{center}
Si tratta di un insieme connesso, ma con interno disconnesso. Inoltre, per tassellare il piano abbiamo bisogno di usare anche la sua immagine riflessa. A seconda quindi di come si decide di definire un prototassello (come classe di isometrie dirette --- ovvero niente riflessioni --- o di isometrie qualsiasi), si può considerare quello di Socolar e Taylor un protoinsieme formato da un singolo prototassello, o da due.

Più di recente, il matematico amatoriale David Smith, ``giocando'' con poliforme, si è imbattuto forse per caso nell'insieme mostrato in grigio nella figura seguente, assieme alla sua immagine riflessa in blu scuro:
\begin{center}
\begin{tikzpicture}

\coordinate (a0) at (0,0);
\coordinate (b0) at ($(120:1)+(180:1)$);
\coordinate (c0) at ($(120:1)+(60:1)$);

\foreach \k in {1,...,7} {
	\coordinate (a\k) at (240+60*\k:1);
	\coordinate (b\k) at ($(b0)+(a\k)$);
	\coordinate (c\k) at ($(c0)+(a\k)$);
	}

\fill[hat]
 (c0) -- ($(a3)!0.5!(a4)$) -- (a3) -- ($(a3)!0.5!(a2)$) -- (a0) -- ($(a4)!0.5!(a5)$) -- (a5) -- (b6) -- ($(b6)!0.5!(b5)$) -- (b0) -- ($(b3)!0.5!(b4)$) -- (b3) -- ($(c5)!0.5!(c4)$) -- cycle;

\draw[gray,very thin] (a1) -- (a2) -- (a3) -- (a4) -- (a5) -- (a6) -- cycle;
\draw[gray,very thin] (b1) -- (b2) -- (b3) -- (b4) -- (b5) -- (b6) -- cycle;
\draw[gray,very thin] (c1) -- (c2) -- (c3) -- (c4) -- (c5) -- (c6) -- cycle;

\foreach \k [evaluate=\k as \knext using int(\k+1)] in {1,...,6} {
	\draw[gray,very thin] (a0) -- ($(a\k)!0.5!(a\knext)$);
	\draw[gray,very thin] (b0) -- ($(b\k)!0.5!(b\knext)$);
	\draw[gray,very thin] (c0) -- ($(c\k)!0.5!(c\knext)$);
	}

\draw[thick]
 (c0) -- ($(a3)!0.5!(a4)$) -- (a3) -- ($(a3)!0.5!(a2)$) -- (a0) -- ($(a4)!0.5!(a5)$) -- (a5) -- (b6) -- ($(b6)!0.5!(b5)$) -- (b0) -- ($(b3)!0.5!(b4)$) -- (b3) -- ($(c5)!0.5!(c4)$) -- cycle;

\end{tikzpicture}
\hspace{1cm}
\begin{tikzpicture}[yscale=-1]

\coordinate (a0) at (0,0);
\coordinate (b0) at ($(120:1)+(180:1)$);
\coordinate (c0) at ($(120:1)+(60:1)$);

\foreach \k in {1,...,7} {
	\coordinate (a\k) at (240+60*\k:1);
	\coordinate (b\k) at ($(b0)+(a\k)$);
	\coordinate (c\k) at ($(c0)+(a\k)$);
	}

\fill[shirt]
 (c0) -- ($(a3)!0.5!(a4)$) -- (a3) -- ($(a3)!0.5!(a2)$) -- (a0) -- ($(a4)!0.5!(a5)$) -- (a5) -- (b6) -- ($(b6)!0.5!(b5)$) -- (b0) -- ($(b3)!0.5!(b4)$) -- (b3) -- ($(c5)!0.5!(c4)$) -- cycle;

\draw[gray,very thin] (a1) -- (a2) -- (a3) -- (a4) -- (a5) -- (a6) -- cycle;
\draw[gray,very thin] (b1) -- (b2) -- (b3) -- (b4) -- (b5) -- (b6) -- cycle;
\draw[gray,very thin] (c1) -- (c2) -- (c3) -- (c4) -- (c5) -- (c6) -- cycle;

\foreach \k [evaluate=\k as \knext using int(\k+1)] in {1,...,6} {
	\draw[gray,very thin] (a0) -- ($(a\k)!0.5!(a\knext)$);
	\draw[gray,very thin] (b0) -- ($(b\k)!0.5!(b\knext)$);
	\draw[gray,very thin] (c0) -- ($(c\k)!0.5!(c\knext)$);
	}

\draw[thick]
 (c0) -- ($(a3)!0.5!(a4)$) -- (a3) -- ($(a3)!0.5!(a2)$) -- (a0) -- ($(a4)!0.5!(a5)$) -- (a5) -- (b6) -- ($(b6)!0.5!(b5)$) -- (b0) -- ($(b3)!0.5!(b4)$) -- (b3) -- ($(c5)!0.5!(c4)$) -- cycle;

\end{tikzpicture}
\end{center}
Tale insieme è formato da 8 aquiloni, ciascuno ottenuto sezionando in sei parti un esagono regolare (esattamente come nella Figura~\ref{fig:alveare}). La storia di questa scoperta è raccontata, ad esempio, in \cite{Kla23}.
La prima figura è detta \emph{cappello}, e la seconda \emph{maglietta}. Il motivo di questi nomi dovrebbe essere chiaro dall'immagine che segue:
\begin{center}
\includegraphics[page=1]{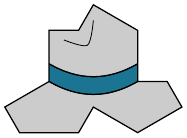}
\hspace{1.2cm}
\includegraphics[page=2]{hatshirt.pdf}
\end{center}
Nonostante l'incredibile semplicità di questi tasselli, è possibile con essi tassellare il piano ed è possibile farlo solo in modo aperiodico \cite{SMKG24a}. Inoltre ogni tassellazione è automaticamente edge-to-edge (senza bisogno di imporre alcuna regola di corrispondenza), a patto di adottare la convenzione di considerare il lato di lunghezza $2$ come due lati di lunghezza unitaria). Un esempio di tassellazione con cappelli e magliette è in Figura~\ref{fig:monotile}.

\begin{figure}[H]
\includegraphics[width=0.8\columnwidth]{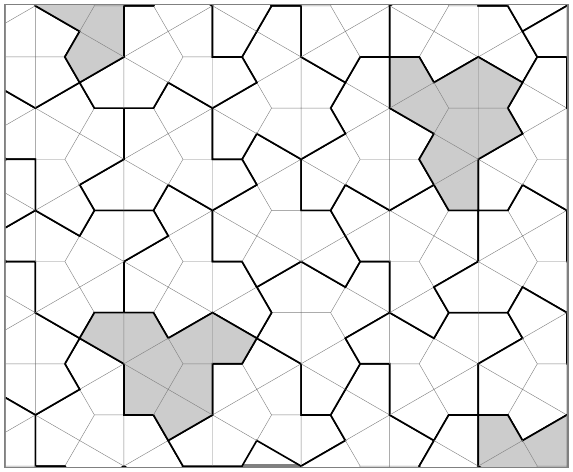}
\caption{Cappelli e magliette}\label{fig:monotile}
\end{figure}

Rispetto al protoinsieme di Socolar e Taylor, qui i tasselli sono omeomorfi ad un disco chiuso (una tassellazione con prototasselli omeomorfi ad un disco chiuso si dice \emph{normale}). Ma anche in questo caso, a seconda di come si definisce un prototassello, il protoinsieme si può considerare formato da un singolo prototassello, oppure da due.

Il cappello ha $14$ lati, alcuni di lunghezza $1$ ed altri di lunghezza $\sqrt{3}$ (se assumiamo che gli esagoni abbiano lato di lunghezza $2$). Il lato di lunghezza $2$ del cappello lo pensiamo come due lati di lunghezza $1$ consecutivi. Per ogni coppia di interi positivi $a$ e $b$ è possibile definire un prototassello, indicato con $\mathtt{Tile}(a,b)$,
mantenendo gli stessi angoli del cappello, ma sostituendo i lati di lunghezza $1$ con lati di lunghezza $a$ ed i lati di lunghezza $\sqrt{3}$ con lati di lunghezza $b$. Chiaramente $\mathtt{Tile}(a,b)$ e $\mathtt{Tile}(ka,kb)$ sono simili, per ogni $a,b,k$ positivi.
Nella figura seguente è mostrato $\mathtt{Tile}(1,1)$, sovrapposto in rosso al cappello:
\begin{center}
\begin{tikzpicture}

\coordinate (a0) at (0,0);
\coordinate (b0) at ($(120:1)+(180:1)$);
\coordinate (c0) at ($(120:1)+(60:1)$);

\foreach \k in {1,...,7} {
	\coordinate (a\k) at (240+60*\k:1);
	\coordinate (b\k) at ($(b0)+(a\k)$);
	\coordinate (c\k) at ($(c0)+(a\k)$);
	}

\begin{scope}[scale=0.5]
\coordinate (v0) at (b6);
\coordinate (v1) at ($(v0)+(2,0)$);
\coordinate (v2) at ($(v1)+(60:1)$);
\coordinate (v3) at ($(v2)+(-30:1)$);
\coordinate (v4) at ($(v3)+(30:1)$);
\coordinate (v5) at ($(v4)+(120:1)$);
\coordinate (v6) at ($(v5)+(-1,0)$);
\coordinate (v7) at ($(v6)+(0,1)$);
\coordinate (v8) at ($(v7)+(150:1)$);
\coordinate (v9) at ($(v8)+(240:1)$);
\coordinate (v10) at ($(v9)+(-1,0)$);
\coordinate (v11) at ($(v10)+(0,-1)$);
\coordinate (v12) at ($(v11)+(210:1)$);
\coordinate (v13) at ($(v12)+(-60:1)$);
\end{scope}


\begin{scope}
\clip
 (c0) -- ($(a3)!0.5!(a4)$) -- (a3) -- ($(a3)!0.5!(a2)$) -- (a0) -- ($(a4)!0.5!(a5)$) -- (a5) -- (b6) -- ($(b6)!0.5!(b5)$) -- (b0) -- ($(b3)!0.5!(b4)$) -- (b3) -- ($(c5)!0.5!(c4)$) -- cycle;

\draw[gray,very thin] (a1) -- (a2) -- (a3) -- (a4) -- (a5) -- (a6) -- cycle;
\draw[gray,very thin] (b1) -- (b2) -- (b3) -- (b4) -- (b5) -- (b6) -- cycle;
\draw[gray,very thin] (c1) -- (c2) -- (c3) -- (c4) -- (c5) -- (c6) -- cycle;

\foreach \k [evaluate=\k as \knext using int(\k+1)] in {1,...,6} {
	\draw[gray,very thin] (a0) -- ($(a\k)!0.5!(a\knext)$);
	\draw[gray,very thin] (b0) -- ($(b\k)!0.5!(b\knext)$);
	\draw[gray,very thin] (c0) -- ($(c\k)!0.5!(c\knext)$);
	}
\end{scope}

\draw[thick]
 (c0) -- ($(a3)!0.5!(a4)$) -- (a3) -- ($(a3)!0.5!(a2)$) -- (a0) -- ($(a4)!0.5!(a5)$) -- (a5) -- (b6) -- ($(b6)!0.5!(b5)$) -- (b0) -- ($(b3)!0.5!(b4)$) -- (b3) -- ($(c5)!0.5!(c4)$) -- cycle;

\fill[red!50,opacity=0.3] (v0) -- (v1) -- (v2) -- (v3) -- (v4) -- (v5) -- (v6) -- (v7) -- (v8) -- (v9) -- (v10) -- (v11) -- (v12) -- cycle;

\draw[red,very thick] (v0) -- (v1) -- (v2) -- (v3) -- (v4) -- (v5) -- (v6) -- (v7) -- (v8) -- (v9) -- (v10) -- (v11) -- (v12) -- cycle;

\end{tikzpicture}
\end{center}
A differenza del cappello, con $\mathtt{Tile}(1,1)$ è possibile tassellare in maniera periodica il piano, se utilizziamo anche la sua immagine riflessa. D'altronde, se non usiamo la sua immagine riflessa, è possibile comunque tassellare il piano, e solo in modo aperiodico. L'idea è allora di deformare leggermente i lati di $\mathtt{Tile}(1,1)$ in modo da ottenere un tassello equivalente, ma che non si riesca ad incastrare con la sua immagine riflessa. Otteniamo il prototassello chiamato
\emph{Spectre}:
\begin{center}
\includegraphics{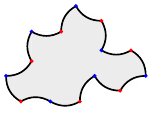}
\end{center}
per la sua forma che ricorda vagamente un fantasma \cite{SMKG24b}. Con questo, sembra che la ricerca di un ``einstein'' sia giunta alla fine. D'altronde, lo studio delle proprietà di questo monotassello aperiodico è solo agli inizi.

\section{Conclusioni}\label{sec:conc}

Nonostante i recenti sviluppi sulle tassellazioni del piano, rimangono ancora molte domande aperte e linee di ricerca interessanti.

Un problema notevole riguarda i protoinsiemi formati da un singolo tassello.
Immaginiamo di costruire una tassellazione del piano edge-to-edge con un singolo prototassello, un poligono regolare con $n$-lati ($n\geq 3$). In ciascun vertice si incontrano $k\geq 2$ poligoni. Siccome l'angolo interno di questi poligoni è $\frac{n-2}{n}\pi$, si deve avere
\[
\frac{2\pi}{k}=\frac{n-2}{n}\pi .
\]
Tale equazione ammette soluzioni $(k,n)$ intere solo per $n=3,4,6$. Il nostro poligono è quindi un triangolo, un quadrato oppure un esagono.

Cosa succede però se rinunciamo all'idea che il nostro prototassello sia un poligono regolare? Un problema tutt'ora aperto è il seguente:

\begin{pro}
Classificare tutti i poligoni convessi che possono apparire in una tassellazione del piano con singolo prototassello.
\end{pro}

Usando la formula di Eulero per grafi planari è possibile dimostrare che 
un tale monotassello può avere al massimo $6$ lati.
La congettura è che esistano solo un numero finito di poligoni convessi con i quali si può tassellare il piano. Il problema, però, è tutt'ora aperto \cite{Zon20}.
Recentemente, Rao ha completato la classificazione dei monotasselli pentagonali \cite{Rao17}, provando che ne esistono esattamente $15$.

Un altro problema facile da formulare, ma di difficile soluzione, è quello che ha motivato la costruzione dei primi prototasselli di Penrose (si veda ad esempio \cite[Cap.~10]{GS87}). Ci poniamo la seguente domanda: cosa succede se spostiamo l'attenzione dalle simmetrie delle tassellazioni alle simmetrie dei singoli tasselli? Ad esempio, cosa possiamo dire su tassellazioni in cui tutti i prototasselli hanno lo stesso tipo di simmetria? Un problema aperto è il seguente:

\begin{pro}
Trovare un esempio di tassellazione del piano in cui tutti i tasselli hanno simmetria 5-fold (o $n$-fold con $n>6$). Oppure dimostrare, sotto ipotesi ragionevoli (ad esempio, prototasselli compatti ed in numero finito), che tale tassellazione non esiste.
\end{pro}

Il lettore interessato ad una panoramica generale sulle tassellazioni può consultare \cite{GS87}, mentre per qualcosa di più recente sulle tassellazioni aperiodiche del piano
si può guardare \cite[Cap.~7]{Sen95}.
Infine, un vasto numero di esempi si può trovare nella Enciclopedia delle Tassellazioni online:
\begin{center}
\web{https://tilings.math.uni-bielefeld.de/}
\end{center}

Per quanto riguarda la dinamica simbolica, abbiamo visto come l'insieme delle successioni di indici di tassellazioni di Penrose formi uno
spazio topologico omeomorfo all'insieme di Cantor. La relazione di ``avere la stessa coda'' è un esempio di relazione ``minimale'' su uno spazio di Cantor, ovvero una relazione le cui classi di equivalenza sono tutte dense nello spazio topologico \cite{Put18}. L'insieme quoziente è il tipico esempio di spazio che si può studiare usando gli strumenti della Geometria Noncommutativa \cite{Con94}.

Numerosi esempi di relazioni minimali su spazi di Cantor si possono costruire a partire da speciali grafi orientati detti \emph{diagrammi di Bratteli}. Sotto alcune ipotesi tecniche (il diagramma deve essere \emph{semplice}) è possibile associare ad un tale diagramma 
una relazione di equivalenza su uno spazio topologico omeomorfo all'insieme di Cantor. Tale relazione ha una topologia naturale che permette di interpretarla come gruppoide étale. Questi gruppoidi possono quindi essere studiati trasformandoli in algebre speciali usando un prodotto di convoluzione. Si ottengono le algebre approssimativamente finito-dimensionali (AF) associate ai diagrammi di Bratteli.

I diagrammi di Bratteli furono introdotti originariamente dal matematico norvegese Ola Bratteli per descrivere C*-algebre che sono limiti induttivi di C*-algebre finito-dimensionali. Giocano un ruolo fondamentale nel teorema di classificazione di George Elliott di tali C*-algebre, e sono stati usati sistematicamente da Anatoly Vershik per la costruzione di subshift in dinamica simbolica.

Per chi fosse interessato alle C*-algebre e alla loro classificazione, una buon testo introduttivo è \cite{Str21}. Per una discussione su relazioni minimali su spazi di Cantor e C*-algebre AF si può guardare \cite[Cap.~6]{DanPen}.

\end{multicols}

\end{document}